\documentclass[10pt, a4paper]{amsart} 
\usepackage{verbatim}
\usepackage[utf8]{inputenc}
\usepackage{eqparbox}

\usepackage{xcolor} 
\usepackage[
  colorlinks=true,   
  linkcolor=blue,   
  citecolor=blue,    
  urlcolor=blue      
]{hyperref}
\usepackage{todonotes}
\usepackage{bigints}
\usepackage[toc]{appendix}
\usepackage{bm}
\usepackage[T1]{fontenc}

\usepackage{graphicx}
\usepackage{enumerate}
\usepackage{amsmath,amsfonts,amssymb}
\usepackage{color}
\def\loc{\operatorname{loc}}
\usepackage{cite}
\usepackage{ latexsym }
\definecolor{citation}{rgb}{0.11,0.67,0.84}
\definecolor{formula}{rgb}{0.1,0.2,0.6}
\definecolor{url}{rgb}{0.11,0.67,0.84}
\usepackage{pgf,tikz}
\usepackage{mathrsfs}
\usepackage{fancyhdr}
\usepackage{dutchcal}

\usepackage{tikz}
\usepackage{amsmath}
\newcommand\FOO{\setcounter{tocdepth}{2}}
\newcommand\BAR{\setcounter{tocdepth}{1}}

\newcommand{\medint}{-\kern -,375cm\int}

\newcommand{\medintinrigo}{-\kern -,315cm\int}
\makeatletter
\newcommand{\linethrough}{\mathpalette\@thickbar}
\newcommand{\@thickbar}[2]{{#1\mkern0mu\vbox{
    \sbox\z@{$#1#2\mkern-0.5mu$}%
    \dimen@=\dimexpr\ht\tw@-\ht\z@+2\p@\relax 
    \hrule\@height0.5\p@ 
    \vskip\dimen@
    \box\z@}}
}
\makeatother

\newtheorem{theorem}{Theorem}[section]
\newtheorem{lemma}[theorem]{Lemma}
\newtheorem{proposition}[theorem]{Proposition}
\newtheorem{remark}{Remark}[section] 

\newtheorem{definition}[theorem]{Definition}
\theoremstyle{plain}  

\numberwithin{equation}{section}
\usepackage{amsthm}

\usepackage{dsfont}

\newcommand{\reqnomode}{\tagsleft@false}

\usepackage{hyperref}

\def\dx{\,{\rm d}x}

\def\dt{\,{\rm d}t}
\def\dy{\,{\rm d}y}

\def \d{\,{\rm d}}

\allowdisplaybreaks
\makeatletter
\DeclareRobustCommand*{\bfseries}{%
  \not@math@alphabet\bfseries\mathbf
  \fontseries\bfdefault\selectfont
  \boldmath
}

\makeatother

\newlength{\defbaselineskip}
\newcommand{\mint}{\mathop{\int\hskip -1,05em -\, \!\!\!}\nolimits}

\newcommand{\R}{\mathbb{R}}

\newcommand{\rr}{\varrho}

\newcommand{\nr}[1]{\lVert #1 \rVert}

\newcommand{\tx}[1]{\textnormal{\texttt{#1}}}

\def\loc{\operatorname{loc}}

\def\eqn#1$$#2$${\begin{equation}\label#1#2\end{equation}}

\newcommand{\sn}{{\sigma}_n}
\newcommand{\tsn}{\tilde{\sigma}_n}

\usepackage{amsmath}

\usepackage{amsmath}

\def\XXint#1#2#3{{\setbox0=\hbox{$#1{#2#3}{\int}$}
     \vcenter{\hbox{$#2#3$}}\kern-.5\wd0}}

\def\XXint#1#2#3{{\setbox0=\hbox{$#1{#2#3}{\int}$ }
		\vcenter{\hbox{$#2#3$ }}\kern-.6\wd0}}

\title{Local boundedness for weak solutions to fractional porous medium equation}
\author[De Filippis]{Filomena De Filippis}  \address{Filomena De Filippis\\Fachbereich Mathematik, Universität Salzburg, Hellbrunner Str. 34, 5020 Salzburg, Austria \\ ORCID ID: 0000-0002-2784-1411}
\email{\url{filomena.defilippis@plus.ac.at}}

\begin{document}

\subjclass[2020]{\vspace{1mm} 35B65, 35R11, 76S05} 

\keywords{\vspace{1mm} boundedness of solutions, porous medium, nonlocal}

\thanks{{\it Acknowledgements.} This research was funded in whole or in part by the Austrian Science Fund (FWF) [10.55776/PAT1850524]. For open access purposes, the author has applied a CC BY public copyright license to any author accepted manuscript version arising from this submission. \\ The author is also supported by INdAM - GNAMPA Project, CUP E53C25002010001.}

\begin{abstract}
We establish local boundedness for solutions to fractional porous medium-type equations in the fast diffusion regime, under optimal tail assumptions.
\end{abstract}

\maketitle
\begin{center}
\begin{minipage}{12cm}
\tableofcontents
\end{minipage}
\end{center}
\section{Introduction}
\noindent
In this paper we investigate boundedness properties of weak solutions to the fractional porous medium-type equation
\begin{equation}\label{pm}
\partial_t u^q - \mathcal{L}_K^s u = 0
\qquad \text{in } \Omega_T,
\end{equation}
where $s\in(0,1)$, $q>1$, and $\Omega_T := \Omega \times (0,T]$ is a space-time cylinder over an open and bounded set $\Omega \subset \mathbb{R}^N$, $N\ge2$. The nonlocal operator $\mathcal{L}_K^s$ is defined as 
\[
-\mathcal{L}^s_K u(x,t)
= 2\mathrm{P.V.}\!\int_{\mathbb{R}^N}
{(u(x,t)-u(y,t))}K(x,y,t)\dy,
\]
where P.V. stands for the principal value and the kernel $K: \R^N \times \R^N \times (0,T] \to [0,\infty)$ is a measurable function such that
$$ \frac{\tx{L}^{-1}}{{|x-y|^{N+2s}}} \leq K(x,y,t) = K(y,x,t) \leq \frac{\tx{L}}{{|x-y|^{N+2s}}},$$
for any $t \in (0,T], (x,y) \in \R^N \times \R^N$ and for some constant $\tx{L} \geq 1$. A wide range of physical models is described by equations of porous medium-type, and more generally by doubly nonlinear parabolic equations. In the local setting, their regularity theory has been extensively developed; we refer the reader to \cite{BDL,BDKS, BDL2,BDMS,BDMS2,KK,KM,KSU} and the references therein for a comprehensive overview. Boundedness of solutions is a significant starting point for the study of higher regularity, such as Harnack inequalities, $C^{0,\alpha}$- or $C^{1,\alpha}$-regularity. In this direction, the first results on boundedness of solutions and gradient regularity for equations of the form
\begin{equation}\label{e22}
\partial_t u^q - \mathrm{div}(|Du|^{p-2}Du)=0
\end{equation}
go back to the seminal works of Ivanov~\cite{I4,I1,I2,I0}. More recently, B\"ogelein \& Duzaar \& Giannazza \& Liao \& Scheven \cite{BDGLS} established gradient H\"older continuity of weak solutions to \eqref{e22} in the fast diffusion regime, and prove their boundedness in the range
\[
0<p-1<q<\frac{N(p-1)+p}{(N-p)_+}=:q_c^p.
\]
The exponent $q_c^p$ plays the role of a critical threshold for boundedness. Indeed, when $p<N$ and $q>q_c^p$, there exist nonnegative local weak solutions to \eqref{e22} in $\mathbb{R}^N\times(-\infty,T)$ which fail to be locally bounded. Thus, $q=q_c^p$ represents the borderline case separating boundedness from blow-up phenomena. In particular, it is shown in \cite{BDGLS} that if $q<q_c^p$, then weak solutions to \eqref{e22} belong to $L^\infty_{\mathrm{loc}}$ and satisfy quantitative estimates. When $q\ge q_c^p$, quantitative bounds can only be derived under the additional qualitative assumption that the solution is locally bounded.

New insights come from \cite{C}, where the authors study local weak sub-solutions to doubly nonlinear equations with double-phase, Orlicz-type, and fully anisotropic operators. Specializing their results to the porous medium equation, corresponding to the case of \eqref{e22} with $p=2$, one finds that at the critical exponent a qualitative $L^\infty_{\mathrm{loc}}$-regularity result holds without any additional a priori assumptions, improving upon earlier results where such regularity was available only under the extra assumption
$u\in L^A_{\mathrm{loc}}(\Omega_T)$ for sufficiently large $A$, see for instance \cite{Eu} and the references therein.

 In the supercritical (and critical) regime, instead, quantitative estimates are obtained provided the solution belongs locally to $L^A$, with
\[
A > \frac{N}{p}(q-p+1), \qquad \text{when } q>p-1.
\]
This condition can be viewed as the parabolic counterpart of the assumptions introduced in \cite{ok} to establish boundedness for solutions to elliptic equations with $(p,q)$-growth, namely
\begin{equation}\label{mg}
q > \frac{Np}{N-p}, \qquad p<N, \qquad 
u \in L^A_{\mathrm{loc}}(\Omega), \qquad 
A > \frac{N}{p}(q-p),
\end{equation}
which are known to be essentially sharp, in the sense that if the bound on $q$ in \eqref{mg} holds and no further regularity is imposed on the solution, counterexamples to boundedness arise; see \cite{M,M1,Gia}. \\  \\ \noindent
The regularity theory for nonlocal parabolic equations has been intensively developed over the last decade. Recent advances have extended the analysis of doubly nonlinear parabolic equations to nonlocal operators, see for example \cite{B1, B2, B3, K1, K2} and references therein. In particular, equations like
$$ \partial_t u^q + (-\Delta)^s_p u = 0,$$
where $(-\Delta)^s_p$ is the fractional $p$-Laplace operator, arises in connection with the global existence of the gradient flow associated with the fractional Sobolev inequality, see \cite{Nak}. Throughout the paper, we work in the fast diffusion regime $q>1$, where solutions exhibit infinite speed of propagation, similarly to the heat equation.

In this context, a central issue is the boundedness of weak solutions to \eqref{pm}, and in particular the delicate interplay with the exponent $q$ in critical regimes. Our analysis relies on a Caccioppoli-type estimate combined with a De Giorgi-type iteration.  While energy estimates of this kind are known for porous medium-type equations of the form \eqref{pm} (see \cite{MY}), the main challenge here is executing the iteration to obtain the $L^\infty$-bound. This requires a precise control of the nonlocal contribution, which introduces a tail condition distinct from the one in \cite{MY}, where the focus was on the expansion of positivity. To address this, we employ a time-dependent truncation method based on the so-called tail,  following the approach introduced in \cite{KW}, see also \cite{L2, He}. This technique allows us to effectively remove the tail contribution from the Caccioppoli estimate. As a consequence, we obtain the first boundedness result under optimal tail assumptions in the setting of fractional porous medium-type equations.\\ \\ \noindent
For a local weak sub-solution to  \eqref{pm} (see Definition \ref{def}), Theorem \ref{th2} below shows that, for $q$ less than or equal to the critical threshold, i.e.
\eqn{ct}
$$ q_c := \frac{N+2s}{N-2s},$$
local boundedness can be obtained without any further integrability requirements, confirming the scenario described in \cite[Theorem 1.1]{C} for the local case.
\begin{theorem} \label{th2}
Let $u$ be a local, nonnegative, weak sub-solution to \eqref{pm} in $\Omega_T$. If
\eqn{qc}
$$
1<q\leq \frac{N+2s}{N-2s},
$$
then $u$ is locally bounded.
\end{theorem}
\noindent
Let us now turn to the quantitative estimate. To this purpose, we introduce the exponent
  \eqn{m}
$$ \tx{m}:=\frac{2}{N}(N+(q+1)s),$$
which arises naturally from the parabolic embedding in Lemma~\ref{em}. For a given $\tx{x}>0$, we also define
\eqn{Lambda}
 $$ \Lambda_{\tx{x}}:= 2s\tx{x} + N(1-q),$$
a quantity that, when evaluated at the solution's natural integrability exponent $q+1$, distinguishes the admissible range for $q$. Our second theorem provides a quantitative $L^\infty$-estimate for local weak sub-solutions, distinguishing between the two regimes $q<q_c$ and $q \geq q_c$, with $q_c$ as in \eqref{ct}.
 \begin{theorem} \label{th1} Assume that $u$ is a local, nonnegative, weak sub-solution to \eqref{pm} in $\Omega_T$.  If $1<q<q_c$, then on any parabolic cylinder $ Q_{\rr,\sigma}(x_0,t_0)  \Subset \Omega_T$, it holds
\begin{align} \label{tesi1}
\operatorname*{ess\,sup}_{Q_{ \rr/2,  \sigma/2}(x_0,t_0)} u \leq c \left ( \frac{{\rr}^{2s}}{\sigma} \right )^{\frac{N}{\Lambda_{\tx{r}}}}\left ( \mint\mint_{Q_{{\rr},\sigma}(x_0,t_0)} u^{\tx{r}}(x,t) \dx \dt\right )^{\frac{2s}{\Lambda_{\tx{r}}}} + c\left ( \frac{\sigma}{\rr^{2s}}\right )^{\frac{1}{q-1}} +  c \textnormal{Tail}^{\frac{1}{q}}(u; Q_{ \rr/2, \sigma}(x_0,t_0)),
\end{align}
for every $\tx{r} \in [1,\tx{m}]$ such that $\Lambda_r >0$, with $\Lambda_\tx{r}$ defined according to \eqref{Lambda}, $\tx{m}$ as in \eqref{m},
and for a constant $c\equiv c(s,q,N,\tx{L})$.
If $ q \geq q_c$, assume moreover that
\eqn{hi}
$$ u \in L^{\tx{r}}_{\loc}(\Omega_T), \  \text{ for some }  \tx{r}> \frac{N}{2s}(q-1).$$
Then, \eqref{tesi1} holds true
on any parabolic cylinder $ Q_{\rr,\sigma} (x_0,t_0)  \Subset \Omega_T$, for a constant $c \equiv c(s,q,N,\tx{L},\tx{r})$.
 \end{theorem}
 \noindent
Remark~\ref{rm1} below shows how, with $\Lambda_\tx{r}>0$, the condition $\tx{r} \le \tx{m}$ corresponds to the regime ensuring local boundedness without further assumptions, whereas $\tx{r} > \tx{m}$ and $q \ge
q_c$ require the additional integrability hypothesis  \eqref{hi}.
\begin{remark}\label{rm1} \normalfont
Since the exponent $\tx{m}$ in \eqref{m} coincides with the exponent appearing in the parabolic embedding Lemma~\ref{em}, the solution satisfies $u \in L^{\tx{m}}_{\loc}(\Omega_T)$. Therefore, if $\tx{r} \le \tx{m}$, the right hand-side of \eqref{tesi1} is finite. Moreover, there exists $\tx{r} \in [1,\tx{m}]$ such  that $\Lambda_{\tx{r}}>0$ if and only if  $\Lambda_{\tx{m}}>0$. Now, observing that
\[
\Lambda_{\tx{m}}=\frac{N+2s}{N}\,\Lambda_{q+1}
\]
we deduce that $\Lambda_{\tx{m}}>0$ is equivalent to $q+1<\tx{m}$ and to $\Lambda_{q+1}>0$, which in turn is equivalent to
\[
q<\frac{N+2s}{N-2s},
\]
namely the critical exponent. If $q \ge q_c$, the quantitative $L^{\infty}$-estimate \eqref{tesi1} can still be obtained under the additional higher integrability assumption \eqref{hi}. Note that in this regime $\tx{r} > q+1 \geq \tx{m}$.
\end{remark}
\noindent
As observed in Remark \ref{rm1}, the behavior of the solution depends on whether the exponent $q$ is subcritical, critical, or supercritical. In the subcritical case, Theorem \ref{th1} adapts the approach of Theorem 1.5 in \cite{BDGLS} to the nonlocal case. In the complementary case, quantitative estimates require assuming higher integrability of the solution a priori, in line with what we have seen before in \cite{C}. 
 \\ \\ \noindent
The paper is organized as follows. In Section~2 we collect the necessary preliminaries and several useful lemmas. In particular, we present the Caccioppoli-type energy estimate, where we already employ a time-dependent truncation level that will serve to balance the nonlocal tail in the De Giorgi iteration. This is precisely the content of Proposition~\ref{pca}. Section~3 is then devoted to the derivation of the  qualitative and quantitative estimates stated in Theorems~\ref{th2} and~\ref{th1}.

 \section{Preliminaries}
\noindent
This section introduces the basic notation and collects several useful results and technical lemmas that will be employed throughout the paper. Moreover, we prove a Caccioppoli-type estimate. \\
\subsection*{Notation}
The symbol $c$ will
denote a generic positive constant, not necessarily the same at each occurrence, that can be determined by the data from the problem under consideration. We use symbols "$\lesssim, \gtrsim$" with subscripts, to indicate that a certain inequality holds up to constants whose relevant dependencies are marked in the suffix. Let $\mathbb{N}_0 = \mathbb{N} \cup \{ 0 \}.$ With $B_R(x_0)$ we denote the open ball of center $x_0 \in \R^N$ and radius $R>0$. When $x_0=0$ we simply write $B_R$. Moreover, we define a general backward parabolic cylinder as
$$
Q_{R,\theta}(x_0, t_0):=B_R(x_0)\times(t_0-\theta,t_0), \qquad (x_0,t_0) \in \R^N \times \R.
$$
For $v \in L^1(\Omega_T, \R^k)$ and a measurable subset $U\subset \Omega$ with positive $\mathcal{L}^N$-measure, we define the slice-wise mean $(v)_U :(0,T) \to \R^k$ of $v$  on $U$ as
$$ (v)_U(t):= \mint_U v(\cdot,t) \dx = \frac{1}{|U|} \int_{U}v(\cdot,t) \dx, \qquad \text{for a.e. } t \in (0,T).$$
Similarly, for a measurable set $Q \subset \Omega_T$ of positive $\mathcal{L}^{N+1}$-measure, we define the mean value $(v)_Q$ of $v$ on $Q$ as
$$ (v)_Q := \mint\mint_Q v \dx\dt=\frac{1}{|Q|} \iint_{Q} v\dx \dt.$$
An essential element in our analysis is the tail \cite{P1,P2} of a function $u(\cdot, t) \in L^1_{\loc}(\R^N,\R)$ with respect to $ Q_{R,\theta}(x_0,t_0)$ defined as:
\eqn{tail}
$$ \textnormal{Tail}(u;  Q_{R,\theta}(x_0,t_0)) := \int_{t_0 - \theta}^{t_0} \int_{\R^N \setminus B_R(x_0)} \frac{|u(x,t)|}{|x-x_0|^{N+2s}} \dx \dt.$$
For $s\in(0,1)$, we introduce the fractional Sobolev space
$W^{s,2}(\R^N)$ defined by
\[
W^{s,2}(\R^N)
:= \left\{
v\in L^2(\R^N) :
\iint_{\R^N \times \R^N}
\frac{|v(x)-v(y)|^2}{|x-y|^{N+2s}}\dx\dy < \infty
\right\},
\]
which is endowed with the norm
\[
\|v\|_{W^{s,2}(\R^N)}
:= \left(\int_{\R^N} |v|^2\dx\right)^{\frac12}
+ \left(\iint_{\R^N \times \R^N}
\frac{|v(x)-v(y)|^2}{|x-y|^{N+2s}}\dx\dy\right)^{\frac12}.
\]
For a domain $\Omega \subset \R^N$, the fractional Sobolev space $W^{s,2}(\Omega)$ can be defined similarly. Moreover, we denote
\[
W^{s,2}_0(\Omega)
:=
\left\{
v \in W^{s,2}(\mathbb{R}^N)
\,:\,
v = 0 \quad \text{almost  everywhere in } \mathbb{R}^N \setminus \Omega
\right\}.
\]
We conclude with the notion of weak super(sub)-solution.
\begin{definition} \label{def}
A measurable function $u: \R^N \times (0,T] \to \R$ satisfying
$$ u \in C_{\loc} (0, T; L^{q+1}_{\loc}( \Omega) ) \cap L^2_{\loc} (0, T; W^{s,2}_{\loc}( \Omega) ), $$
is a local, {weak super(sub)-solution} to the equation \eqref{pm} in $\Omega_T$, if for every sub-interval $[t_1,t_2] \subset (0,T]$, every bounded open set $\tilde{\Omega} \subset \Omega$, it holds
\eqn{st}
$$
\int_{t_1}^{t_2} \int_{\R^N} \frac{|u(x,t)|}{1+|x|^{N+2s}} \dx \dt < \infty
$$
and 
\begin{align}  \label{testfunction}
& \int_{{\tilde{\Omega}}} |u|^{q-1}u \phi \dx \big |_{t_1}^{t_2}-
\int_{t_1}^{t_2}\int_{ {\tilde{\Omega}}}  |u|^{q-1}u \, \partial_t\phi \dx \dt \notag \\ & \qquad \qquad + \int_{t_1}^{t_2} \iint_{\R^N \times \R^N}(u(x,t)-u(y,t))(\phi(x,t)-\phi(y,t)) K(x,y,t) \dx \dy \dt \geq (\leq) \, 0
\end{align}
is satisfied for all nonnegative test functions
\[
\phi \in  L^2_{\loc} (0, T; W^{s,2}_0({\tilde{\Omega}}) ) \cap  W^{1,q+1}_{\loc} (0, T; L^{q+1}({\tilde{\Omega}})) .
\]
A function that is both a weak super- and sub-solution is called a weak solution.
\end{definition}
\noindent
Note that \eqref{st} ensures the finiteness of the tail \eqref{tail} of the super(sub)-solution, and this finiteness, together with the fact that the test function belongs to $W^{1,1}_{\loc}(0,T; L^1(\tilde{\Omega}))$, guarantees that the fractional term in \eqref{testfunction} is finite.


\subsection*{Useful lemmas}
\noindent
For $u,k\in\R$ and $q>0$, we introduce
\eqn{G}
$$
\mathfrak{g}_{\pm}(u,k):= \pm q\int_k^u |t|^{q-1}(t-k)_{\pm}\dt,
$$
where
$$
(t-k)_+:=\max\{t-k,0\},\qquad (t-k)_-:=\max\{-(t-k),0\}.
$$
The following lemma holds, see \cite[Lemma 2.2]{BDL}.
\begin{lemma} \label{2.3}
Let $q > 0$ and consider the function $\mathfrak{g}_{\pm}$ defined in \eqref{G}. Then there exists a constant $c \equiv c(q) > 0$ such that for all $a, b \in \mathbb{R}$ it holds
    \[
        \frac{1}{ c} \big( |a| + |b| \big)^{q-1} (a - b)_{\pm}^2 
        \le \mathfrak{g}_{\pm}(a, b) 
        \le  c \big( |a| + |b| \big)^{q-1} (a - b)_{\pm}^2.
    \]  
\end{lemma}
\noindent
Let us now state the fast geometric convergence Lemma, see \cite[Chapter I, Lemma 4.1]{D}.
\begin{lemma}\label{fg}
    Let $(Y_n)_{n \in \mathbb{N}_0}$ be a sequence of positive real numbers satisfying the recursive inequalities
    \[
        Y_{n+1} \le c \, b^n \, Y_n^{1 + \alpha}
    \]
    where $c, b > 1$ and $\alpha > 0$ are given numbers. 
    If 
    \[
        Y_0 \le c^{-1/\alpha} \, b^{-1/\alpha^2},
    \]
    then $Y_n \to 0$ as $n \to \infty$.
\end{lemma}
\noindent
Next, the basic iteration Lemma \cite[Lemma 6.1]{G}.
\begin{lemma}\label{il}
    For $r < R$, let $h \colon [r, R] \to \R$ be a nonnegative bounded function such that
    \[
        h(r_1) \le \vartheta \, h(r_2) 
        + \frac{A}{(r_2 - r_1)^\alpha} 
        + \frac{B}{(r_2 - r_1)^\beta} 
        + C
    \]
    for all $r < r_1 < r_2 < R$,  where $A, B, C \ge 0$, $\alpha > \beta \ge 0$ and $\vartheta \in (0, 1)$. 
    Then there exists a constant $c \equiv c(\alpha, \vartheta) > 0$ such that
    \[
        h(r) \le c(\alpha, \vartheta) 
        \left( \frac{A}{(R - r)^\alpha} 
        + \frac{B}{(R - r)^\beta} 
        + C \right).
    \]
\end{lemma}
\noindent
We conclude this section with the parabolic Sobolev embedding lemma.
\begin{lemma}[Parabolic embedding]\label{em}
Let $s \in (0,1)$ and $\tx{m}$ be as in \eqref{m}. For any function
\[
u \in L^2\big(t_1,t_2;W^{s,2}(B_R)\big) \cap L^\infty\big(t_1,t_2;L^{q+1}(B_R)\big),
\]
which is compactly supported in $B_{(1-d)R}$ for some $d \in (0,1)$ and for almost every $t \in (t_1, t_2)$, there holds
\[
\begin{aligned}
&\int_{t_1}^{t_2} \int_{B_R} |u(x,t)|^{\tx{m}}   \dx   \dt \\
&\qquad \leq c \Bigg( R^{2s} \int_{t_1}^{t_2} \iint_{B_R \times B_R}
\frac{|u(x,t)-u(y,t)|^2}{|x-y|^{N+2s}}   \dy   \dx   \dt 
+ \frac{1}{d^{N+2s}} \int_{t_1}^{t_2} \int_{B_R} |u(x,t)|^2   \dx   \dt \Bigg) \\
&\qquad \qquad\times \Bigg( \operatorname*{ess\,sup}_{t_1 < t < t_2} \mint_{B_R} |u(x,t)|^{q+1}   \dx \Bigg)^{\frac{2s}{N}},
\end{aligned}
\]
for a positive constant $c \equiv c(s,q,N)$.
\end{lemma}
\begin{proof}
The proof follows by a slight modification of the argument in
\cite[Proposition A.3]{L1}. Precisely, setting \[
\kappa_*:=\frac{N}{N-2s},
\]
we observe that
\[
\tx{m}=\frac{2}{N}\bigl(N+(q+1)s\bigr)
  =2+(q+1)\frac{\kappa_*-1}{\kappa_*}.
\]
We apply H\"older's inequality with exponents
\( \left (
\kappa_*,
\frac{\kappa_*}{\kappa_*-1} \right )
\)
and we obtain, for almost every $t\in(t_1,t_2)$,
\[
\int_{B_R}|u(x,t)|^\tx{m}\,dx
\le
\left(
\int_{B_R}|u(x,t)|^{2\kappa_*}\,dx
\right)^{\frac{1}{\kappa_*}}
\left(
\int_{B_R}|u(x,t)|^{q+1}\,dx
\right)^{\frac{2s}{N}}.
\]
Applying \cite[Proposition A.2]{L1} with $p=2$ to the first
factor and integrating over $(t_1,t_2)$, the conclusion follows.
\end{proof}
\addtocontents{toc}{\protect\FOO}
\subsection{Caccioppoli estimate} 
Weak solutions to parabolic equations generally lack a time derivative in the Sobolev sense, which prevents their direct use in testing functions. To overcome this difficulty, one commonly employs suitable mollifications in the time variable. Then, we introduce for any $v\in L^1(Q_{R,\theta})$ and $h>0$, the mollified functions
\eqn{mf}
$$
(v)_h(x,t):=\frac1h\int_{-\theta}^t e^{\frac{\tau-t}{h}}v(x,\tau)\, \d\tau, \qquad
[v]_h(x,t):=\frac1h\int_t^0 e^{\frac{t-\tau}{h}}v(x,\tau)\,\d\tau.
$$
The following convergence lemma for mollified functions holds.
\begin{lemma} \label{convlem}
For any $\phi \in C(-\theta,0; L^{q+1}(B_R))$ and any $\delta \in (0,\theta)$ it holds
$$ \sup_{t \in [-\theta + \delta,0]} \nr{(\phi)_h(t)-\phi(t)}_{L^{q+1}(B_R)} \overset{h \to 0}{\longrightarrow} 0.$$
The same conclusion holds for $[\phi]_h$ in the interval $[-\theta,-\delta]$.
\end{lemma}
\begin{proof}
For any $(x,t) \in B_R \times [-\theta,0]$, using the definition \eqref{mf}, it holds
    \begin{align*}
        (\phi)_h (x,t) -\phi(x,t) & = \frac{1}{h} \int_{-\theta}^t e^{\frac{\tau-t}{h}} ( \phi(x,\tau)-\phi(x,t)) \d \tau + \phi(x,t) \left (\frac{1}{h}\int_{-\theta}^t e^{\frac{\tau-t}{h}} \d \tau -1 \right ) \notag \\ & = \frac{1}{h} \int_{-\theta}^t e^{\frac{\tau-t}{h}} ( \phi(x,\tau)-\phi(x,t)) \d \tau - e^{\frac{-\theta-t}{h}}\phi(x,t),
    \end{align*}
    and so, for $t \in [-\theta + \delta,0]$, it holds
    \begin{align} \label{nes1}
       \nr{(\phi)_h (t) -\phi(t)}_{L^{q+1}(B_R)} & \leq \frac{1}{h} \int_{-\theta}^t e^{\frac{\tau-t}{h}} \nr{ \phi(\tau)-\phi(t)}_{L^{q+1}(B_R)} \d \tau + e^{\frac{-\delta}{h}} \sup_{t \in [-\theta,0]}\nr{\phi(t)}_{L^{q+1}(B_R)}.
    \end{align}
    Now, call $\textnormal{I}$ the integral in the right hand-side above. With the following change of variable $s=\frac{\tau-t}{h}$, we estimate it as
    \begin{align} \label{nes2}
        \textnormal{I} &\leq \int_{\frac{-\theta-t}{h}}^{\frac{-\theta-t}{\sqrt{h}}} e^{s}\nr{\phi(t+sh)-\phi(t)}_{L^{q+1}(B_R)} \d s + \int_{\frac{-\theta-t}{\sqrt{h}}}^{0} e^{s} \nr{\phi(t+sh)-\phi(t)}_{L^{q+1}(B_R)} \d s \notag \\ & \leq 2 e^{\frac{-\delta}{\sqrt{h}}} \sup_{t \in [-\theta,0]} \nr{\phi(t)}_{L^{q+1}(B_R)} + \sup_{s \in \left ( \frac{-\theta-t}{\sqrt{h}},0 \right )} \nr{\phi(t+sh)-\phi(t)}_{L^{q+1}(B_R)}.
    \end{align}
Since $\phi \in  C(-\theta,0; L^{q+1}(B_R))$, by \eqref{nes1}, \eqref{nes2}, for any $\varepsilon >0$, we choose $h_0>0$ such that $ \displaystyle \sup_{t \in [-\theta+\delta,0]} \nr{(\phi)_h (t) -\phi(t)}_{L^{q+1}(B_R)}< \varepsilon$ holds for any positive $h<h_0$. For $[\phi]_h$ the calculation are similar, see \cite[Lemma B.1]{MY}. The proof is complete.  
\end{proof}
\noindent
We are ready to prove the Caccioppoli inequality for weak sub-solutions to \eqref{pm} in $ \Omega_T$.
\begin{proposition}[Caccioppoli estimate] \label{pca}
Let $q>1$. Assume that $u$ is a local, nonnegative, weak sub-solution to \eqref{pm} in $\Omega_T$. Let $0<\sigma_1<\sigma_2<\theta$ and $0<r<\rr<\tilde{\rr}<R$ be such that $ Q_{R,\theta}(x_0,t_0)\Subset  \Omega_T$. For any $k>0$, define $k(t):=k+\ell(t)$, where
\begin{equation}\label{ell}
\ell(t):=\left[2\tx{L}\left(\frac{R}{R-\tilde{\rr}}\right)^{N+2s}\int_{t_0-\theta}^{t}\int_{\R^N\setminus B_R(x_0)}\frac{u(x,\tau)}{|x-x_0|^{N+2s}}\dx\,\mathrm{d}\tau\right]^{\frac1q}.
\end{equation}
Then,
\begin{align}\label{caccioppoli}
&\sup_{t\in(t_0-\sigma_1,t_0)}\int_{B_r(x_0)}\mathfrak{g}_{+}(u(x,t),k(t))\dx \notag
\\ & \qquad +\int_{t_0-\sigma_1}^{t_0}\!\!\iint_{B_r(x_0)\times B_r(x_0)}\frac{|(u(x,t)-k(t))_{+}-(u(y,t)-k(t))_{+}|^2}{|x-y|^{N+2s}}\dx\dy\dt \notag\\
&\qquad+\iint_{Q_{r,\sigma_1}(x_0,t_0)}(u(x,t)-k(t))_{+}\Big(\int_{B_r(x_0)}\frac{(u(y,t)-k(t))_{-}}{|x-y|^{N+2s}}\dy\Big)\dx\dt \notag\\
& \quad \le\frac{c}{\sigma_2-\sigma_1}\iint_{Q_{\rr,\sigma_2}(x_0,t_0)}\mathfrak{g}_{+}(u(x,t),k(t))\dx\dt
+c\frac{\rr^{2(1-s)}}{(\rr-r)^2}\iint_{Q_{\rr,\sigma_2}(x_0,t_0)}(u(x,t)-k(t))_{+}^2\dx\dt \notag\\
& \qquad + c R^N \left ( \frac{1}{\rr-r}\right )^{N+2s}\sup_{t \in [t_0-\theta,t_0]} \mint_{B_R(x_0)} (u(y,t)-k(t))_{+} \dy \iint_{Q_{\rr,\sigma_2}(x_0,t_0)}  (u(x,t)-k(t))_{+} \dx  \dt,
\end{align}
for a positive constant $c \equiv c(s,N,\tx{L})$.
\end{proposition}
\noindent
The proof follows the general scheme of \cite[Lemma 4.3]{MY}. 
The main difference concerns the truncation level: instead of a fixed constant $k>0$, 
we introduce a time-dependent truncation $k=k(t)$, which is explicitly chosen and is 
absolutely continuous on any compact set of $(t_0-\theta,t_0]$. This choice produces an additional term in the energy 
estimate coming from the time derivative of $k(\cdot)$. Such a term is arranged to exactly 
compensate the nonlocal contribution arising from the tail.
\begin{proof}[Proof of Proposition \ref{pca}]
Let us highlight the main differences with respect to \cite[Proof of Lemma 4.3]{MY}. Let $u$ be as in the statement of the theorem. Take for simplicity $(x_0,t_0)=(0,0)$. We divide the proof into four steps.
\subsubsection*{Step 1: test function}
Let $\sigma_{1}, \sigma_{2}, r,\rr,\tilde{\rr}$  be as in the statement and choose any  $\tau_0\in (-\sigma_{1}, 0]$. Consider the time dependent cut-off functions $\zeta, \psi_\varepsilon$,
$0< \varepsilon < \tau_0 + \sigma_{1}$, defined as follows
\[
\zeta(t) :=
\begin{cases} 
0 & t \in (-\infty,-\sigma_{2}],\\ 
\displaystyle \frac{t +\sigma_{2}}{\sigma_{2}-\sigma_{1}} & t \in [-\sigma_{2},-\sigma_{1}],\\ 
1 & t \in [-\sigma_{1}, \infty),
\end{cases}
\qquad
\psi_{\varepsilon}(t) := 
\begin{cases} 
1 & t \in (-\infty, \tau_0- \varepsilon],\\ 
\displaystyle \frac{\tau_0-t}{\varepsilon} & t \in [\tau_0- \varepsilon, \tau_0],\\ 
0 & t \in [\tau_0, \infty),
\end{cases}
\]
and the space dependent cut-off function  $\eta \in C_{0}^{\infty}(B_R)$ defined by
\[
\eta (y) = \left\{ \begin{array}{ll}1 & y\in B_r,\\ 0 & y\in B_R\backslash B_{\frac{r + \varrho}{2}}, \end{array} \right.
\]
with $0\leq \eta \leq 1$, $|D\eta |\leq \frac{2}{\varrho - r}$ in $B_R$. For $k(\cdot)$ as in the statement of the theorem, denote with $w_{\pm}(x,t):= (u(x,t) - k(t))_{\pm}$, and choose the admissible test function:
\[
\phi := \eta^{2}\zeta \psi_{\varepsilon}w_{+}.
\]
Therefore, by \cite[Lemma B.2]{MY}, bearing in mind the support of $\phi$, we get
\begin{align} \label{ma}
& \iint_{Q_{R,\theta}}\partial_t(u^q)_{h}\eta^{2}\zeta \psi_{\varepsilon}w_{+}\dx \dt \notag \\ & \qquad
+\int_{{-\theta}}^{0}\dt \iint_{\mathbb{R}^{N}\times \mathbb{R}^{N}}{(u(x,t) - u(y,t))}K(x,y,t)\big([\eta^{2}\zeta \psi_{\varepsilon}w_{+}]_{h}(x,t) - [\eta^{2}\zeta \psi_{\varepsilon}w_{+}]_{h}(y,t)\big)\dx \dy \notag \\
&  \quad\leq- \int_{B_R} u^q(t)[\eta^{2}\zeta \psi_{\varepsilon}w_{+}]_{h}(t)\dx \Big|_{t=0}^{t=-\theta} = 0.
\end{align} 
\subsubsection*{Step 2: limit for $h, \varepsilon\to 0$}
Call $\mathrm{L}_1$ and $\mathrm{L}_2$ the two terms in the left-hand side of \eqref{ma}. We observe that,
\begin{align} \label{L1}
\mathrm{L}_1 = &\iint_{Q_{R,\theta}}\partial_t(u^q)_{h}\,\eta^{2}\zeta \psi_{\varepsilon}\big[( u(x,t)-k(t))_{+} - \big( (u^q)_{h}^{\frac{1}{q}}(x,t)-k(t)\big)_{+}\big] \dx \dt \nonumber \\
& \qquad +\iint_{Q_{R,\theta}}\partial_t(u^q)_{h}\,\eta^{2}\zeta \psi_{\varepsilon}\big( (u^q)_{h}^{\frac1q}(x,t)-k(t)\big)_{+} \dx \dt =: \mathrm{L}_{1,1} + \mathrm{L}_{1,2}.
\end{align}
Following \cite{MY}, the term $\mathrm{L}_{1,1}$  is nonnegative. On the other hand, the term $\mathrm{L}_{1,2}$ gives rise to an additional contribution due to the time dependence of  $k(\cdot)$, i.e.,
\begin{align} \label{e1}   
 \mathrm{L}_{1,2} & :=\iint_{Q_{R,\theta}}\eta^2\zeta\psi_\varepsilon\,\partial_t\mathfrak{g}_+\big((u^q)_h^{\frac 1q}(x,t),k(t)\big)\dx\dt \notag \\ & \qquad +\iint_{Q_{R,\theta}}\eta^2\zeta\psi_\varepsilon k'(t)\big((u^q)_h(x,t)-k^q(t)\big)_+\dx\dt=:\mathrm{L}_{1,2}'+\mathrm{L}_{1,2}''.
\end{align}
Observing that $k(\cdot)$ is absolutely continuous in $[-\sigma_2,0]$, the term $\mathrm{L}_{1,2}'$ can be handled as in the cited proof, namely
\eqn{nm}
$$ \lim_{h,\varepsilon \to 0}\mathrm{L}_{1,2}' \geq \int_{B_r} \mathfrak{g}_+(u(x,\tau_0),k(\tau_0)) \dx - \frac{1}{\sigma_2 -\sigma_1} \iint_{Q_{\rr,\sigma_2}} \mathfrak{g}_+(u(x,t),k(t)) \dx\dt.$$
Now, by \cite[Lemma B.1]{MY}, it holds $(u^q)_h \to u^q$ in $L^1(Q_{R,\theta})$ as $h \to 0$ and $\nr{(u^q)_h}_{L^1(Q_{R,\theta})} \leq \nr{u^q}_{L^1(Q_{R,\theta})}$. Moreover, by definition $\psi_\varepsilon \to 1$ on its support, as $\varepsilon \to 0$. Therefore, letting $h, \varepsilon \to0$, we obtain
\eqn{e2}
$$
\mathrm{L}_{1,2}''\to\int_{-\theta}^{\tau_0}\int_{B_R}\eta^2\zeta k'(t)\big(u^q(x,t)-k^q(t)\big)_+\dx\dt.
$$
Moreover, for all $(x,t) \in Q_{R,\theta}$,
\eqn{e3}
$$
k'(t)\big(u^q(x,t)-k^q(t)\big)_+\ge qk'(t)k^{q-1}(t)(u(x,t)-k(t))_+=(k^q(t))'(u(x,t)-k(t))_+,
$$
and, by \eqref{ell},
\eqn{e4}
$$
(k^q(t))'\geq2\tx{L}\left(\frac{R}{R-\tilde{\rr}}\right)^{N+2s}\int_{\R^N\setminus B_R}\frac{u(x,t)}{|x|^{N+2s}}\dx.
$$
Hence, by \eqref{L1}-\eqref{e4}, we have
\begin{align}\label{T1}
\lim_{h,\varepsilon \to 0}\mathrm{L}_{1} \geq & \int_{B_r}\mathfrak{g}_+(u(x,\tau_0),k(\tau_0)) \dx - \frac{1}{\sigma_2 -\sigma_1} \iint_{Q_{\rr,\sigma_2}} \mathfrak{g}_+(u(x,t),k(t)) \dx\dt  \notag \\ & \qquad +2\tx{L}
\left(\frac{R}{R-\tilde{\rr}}\right)^{N+2s}\int_{-\sigma_2}^{\tau_0}\!\!\int_{B_\rr}\eta^2\zeta w_{+}(x,t)\dx\dt \left (\int_{\R^N\setminus B_R}\frac{u(y,t)}{|y|^{N+2s}}\dy\right ).
\end{align}
We now consider the fractional term $\mathrm{L}_2$. The term $\mathrm{R}_{2,2}$, which appears in \cite[(p.~2619)]{MY}, is estimated here in a different way:  
\begin{align*}
\text{R}_{2,2} & := \int_{-\theta}^0 \left (\int_{B_\rr^c} \frac{|u(y,t)|}{|y|^{N+2s}} \dy \right ) \left ( \int_{B_{\frac{r+\rr}{2}}} |(\phi)_h -\phi| \dx \right ) \dt  \\ & \leq \int_{-\theta}^0 \int_{B_\rr^c} \frac{|u(y,t)|}{|y|^{N+2s}} \dy \sup_{t \in [-\theta,0]} \nr{(\phi)_h-\phi}_{L^1(B_{(r+\rr)/2})} \\ & = \int_{-\theta}^0 \int_{B_\rr^c} \frac{|u(y,t)|}{|y|^{N+2s}} \dy \sup_{t \in [-\sigma_2,0]} \nr{(\phi)_h-\phi}_{L^1(B_{(r+\rr)/2})},
\end{align*}
where for the last equality we used the support of $\phi$. Now, since $\sigma_2 < \theta$, there exists $\delta >0$ such that $-\theta +\delta < -\sigma_2$. Therefore, by Lemma \ref{convlem}, we get that $\text{R}_{2,2} \to 0$ as $h \to 0$. Recalling the fact that $\phi \in C(-\theta,0; L^{q+1}(B_R))$, the remainder of the proof proceeds as in \cite{MY}. Therefore, passing to the limit $h,\varepsilon \to 0$, it holds
$$ \mathrm{L}_2 \to \int_{-\sigma_2}^{\tau_0} \iint_{\mathbb{R}^N\times\mathbb{R}^N} (u(x,t)-u(y,t)){(w_{+}(x,t)\eta^2(x) - w_{+}(y,t)\eta^2(y))\,\zeta(t)} K(x,y,t)\dx\dy \dt.$$
\subsubsection*{Step 3: estimation of the fractional term}
From the previous step we obtained
\begin{align} \label{above}
  &  \int_{B_r}\mathfrak{g}_+(u(x,\tau_0),k(\tau_0)) \dx  \notag \\ & \qquad +
2\tx{L}\left(\frac{R}{R-\tilde{\rr}}\right)^{N+2s}\int_{-\sigma_2}^{\tau_0}\!\!\int_{B_\rr}\eta^2\zeta w_{+}(x,t)\dx\dt\left (\int_{\R^N\setminus B_R}\frac{u(y,t)}{|y|^{N+2s}}\dy  \right ) \notag \\ & \qquad +  \int_{-\sigma_2}^{\tau_0} \iint_{\mathbb{R}^N\times\mathbb{R}^N} (u(x,t)-u(y,t)){(w_{+}(x,t)\eta^2(x) - w_{+}(y,t)\eta^2(y))\,\zeta(t)}K(x,y,t)\dx\dy \dt \notag  \\ & \quad \leq \frac{1}{\sigma_2 -\sigma_1} \iint_{Q_{\rr,\sigma_2}} \mathfrak{g}_+(u(x,t),k(t)) \dx\dt.
\end{align}
Denote with $\mathrm{L}_2'$ the fractional term in the left-hand side of \eqref{above}, and with $\mathrm{I}_2'$ the integrand, we observe that
$$ \mathrm{L}_{2}'= \int_{-\sigma_2}^{\tau_0} \iint_{B_\rr \times B_\rr} \mathrm{I}_2' \dx \dy \dt + 2\int_{-\sigma_2}^{\tau_0} \iint_{B_\rr^c \times B_\rr} \mathrm{I}_2' \dx \dy \dt =: \mathrm{L}_{2,1}' + \mathrm{L}_{2,2}'.$$
As regards $\mathrm{L}_{2,1}'$, we estimate it as in \cite[inequality (B.29)]{MY} (see also \cite[Proposition 2.1]{L1}),
\begin{align} \label{l11}
\mathrm{L}_{2,1}^{\prime}\geq & c\int_{-\sigma_1}^{\tau_0} \iint_{B_r\times B_r}\frac{|w_{+}(x,t) - w_{+}(y,t)|^{2}}{|x - y|^{N + 2s}}\dx\dy \dt \nonumber \\
& +c\int_{-\sigma_1}^{\tau_0}\iint_{B_r\times B_r} w_+(y,t)\frac{w_{-}(x,t)}{|x-y|^{N+2s}}\dx \dy \dt \nonumber \\
& -c\int_{-\sigma_2}^{\tau_0}  \iint_{B_\varrho\times B_\varrho}\max \left\{w_{+}(x,t),w_{+}(y,t)\right\}^{2}\frac{|\eta (x) - \eta (y)|^{2}}{|x-y|^{N+2s}}\dx \dy \dt,
\end{align}
for a positive constant $c \equiv c(\tx{L})$. While for $\mathrm{L}_{2,2}'$, using the fact that $\eta(y)=0$ for $y \in B_\rr^c$, the inequality
\[
 (u(y, t)- u(x, t)  )  (u(x, t) - k  )_+ \leq  (u(y, t) - k  )_+  (u(x, t) - k  )_+,
\]
and
\[
\frac{|y-x|}{|y|}\ge \frac{R-\tilde{\rr}}{R} \quad \text{if } |x|\le \tilde{\rr},\,|y|\ge R, \qquad
\frac{|y-x|}{|y|}\ge \frac{\rr-r}{2\rr} \quad \text{if } |x|\le \frac{r+\rr}{2},\,|y|\ge \rr,
\]
we obtain
    \begin{align} \label{T2}
       - \text{L}_{2,2}'& \leq 2\tx{L}\int_{-\sigma_2}^{\tau_0} \dt \iint_{B_\rr \times B_\rr^c}\eta^2\zeta w_{+}(x,t) \frac{w_{+}(y,t)}{|x-y|^{N+2s}} \dx \dy \notag \\ & \leq 2\tx{L} \int_{-\sigma_2}^{\tau_0}\int_{B_\rr} \eta^2\zeta w_{+}(x,t) \dx \dt \left (\displaystyle\operatorname*{ ess\, sup}_{x \in B_{\frac{r+\rr}{2}}} \int_{B_\rr^c} \frac{w_{+}(y,t)}{|x-y|^{N+2s}} \dy\right ) \notag \\ & =2\tx{L}  \int_{-\sigma_2}^{\tau_0}\int_{B_\rr}\eta^2\zeta w_{+}(x,t) \dx \dt  \displaystyle \cdot \operatorname*{ ess \, sup}_{x \in B_{\frac{r+\rr}{2}}} \left ( \int_{\R^N \setminus B_R} \frac{w_{+}(y,t)}{|x-y|^{N+2s}} \dy \right . \notag \\ & \left . \qquad + \int_{B_R \setminus B_\rr} \frac{w_{+}(y,t)}{|x-y|^{N+2s}} \dy \right ) \notag \\ & \leq 2\tx{L}  \int_{-\sigma_2}^{\tau_0}\int_{B_\rr} \eta^2\zeta w_{+}(x,t) \dx  \dt \left [ \left (\frac{R}{R-\tilde{\rr}} \right )^{N+2s} \int_{\R^N \setminus B_R} \frac{u(y,t)}{|y|^{N+2s}} \dy  \right . \notag  \\ & \left . \qquad  +\left ( \frac{2}{\rr-r}\right )^{N+2s} \int_{B_R \setminus B_\rr} w_{+}(y,t) \dy\right ].
    \end{align}
\subsubsection*{Step 4: derivation of the final estimate} Hence, using \eqref{l11} and \eqref{T2} in \eqref{above}, we obtain the following estimate in which the contribution of the tail is no longer present:
\begin{align} \label{nes}
  &  \int_{B_r}\mathfrak{g}_+(u(x,\tau_0),k(\tau_0)) \dx  +\int_{-\sigma_1}^{\tau_0} \iint_{B_r\times B_r}\frac{|w_{+}(x,t) - w_{+}(y,t)|^{2}}{|x - y|^{N + 2s}}\dx\dy \dt \notag \\ & \qquad+\int_{-\sigma_1}^{\tau_0}\iint_{B_r\times B_r} w_+(y,t)\frac{w_{-}(x,t)}{|x-y|^{N+2s}}\dx \dy \dt  \notag \\ & \quad  \leq \frac{1}{\sigma_2 -\sigma_1} \iint_{Q_{\rr,\sigma_2}} \mathfrak{g}_+(u(x,t),k(t)) \dx\dt \notag \\ & \qquad +  \frac{c}{(\rr-r)^2} \frac{\rr^{2(1-s)}}{2(1-s)} \iint_{Q_{\rr,\sigma_2}} w_{+}(x,t)^2 \dx \dt \notag \\ & \qquad + c R^N \left ( \frac{1}{\rr-r}\right )^{N+2s}\sup_{t \in [-\theta,0]} \mint_{B_R} w_{+}(x,t)\dx   \iint_{Q_{\rr,\sigma_2}} w_{+}(y,t) \dy \dt,
\end{align}
with $c \equiv c(s,N,\tx{L})$; note that we made a further estimation of the last term in \eqref{l11} following  \cite[inequality (B.34)]{MY}, and using  the fact that $|\eta(y) - \eta(z)| \lesssim \frac{1}{\varrho - r} |y - z|$. Now, in the estimate above, choose first $\tau_0= 0$ to make the second and the third term on the left-hand side of \eqref{nes} bounded by those of the right-hand side. Next, let $\tau_0$ be any number in $(-\sigma_1,0)$. Dropping the second and third nonnegative terms in the left-hand side of \eqref{nes} and taking the supremum on $ \tau_0 \in (-\sigma_1,0)$, we get that the first integral of \eqref{nes} is bounded by the right-hand side. Therefore, we get
\begin{align*} 
  &  \sup_{\tau_0 \in (-\sigma_1,0)}\int_{B_r}\mathfrak{g}_+(u(x,\tau_0),k(\tau_0)) \dx   + \iint_{Q_{r,\sigma_1}}\int_{B_r} w_{+}(x,t) \left (\int_{B_r} \frac{w_-(y,t)}{|x-y|^{N+2s}} \dy \right )\dx  \dt \nonumber \\ & \qquad + \int_{-\sigma_1}^{0} \iint_{B_r\times B_r}\frac{|w_{+}(x,t) - w_{+}(y,t)|^{2}}{|x - y|^{N + 2s}}\dx\dy \dt  \notag \\ & \quad \leq \frac{c}{\sigma_2 -\sigma_1} \iint_{Q_{\rr,\sigma_2}} \mathfrak{g}_+(u(x,t),k(t)) \dx\dt \notag \\ & \qquad +\frac{c}{(\varrho -r)^2}\int_{-\sigma_2}^{0}  \iint_{B_\varrho\times B_\varrho}\frac{\max \left\{w_{+}(x,t),w_{+}(y,t)\right\}^{2}}{|x-y|^{N+2s-2}}\dx \dy \dt \notag \\ & \qquad + c R^N \left ( \frac{1}{\rr-r}\right )^{N+2s}\sup_{t \in [-\theta,0]} \mint_{B_R} w_{+}(x,t)\dx   \iint_{Q_{\rr,\sigma_2}} w_{+}(y,t) \dy \dt,
\end{align*}
the Caccioppoli estimate \eqref{caccioppoli} is proved.
\end{proof}
\addtocontents{toc}{\protect\BAR}

\section{$L^\infty$-estimates}
\noindent
This section is devoted to establishing the local boundedness of weak sub-solutions. Theorem \ref{th1} provides a quantitative estimate when $q$ is below the critical exponent, whereas in the critical or supercritical case higher integrability of the solution is required. At the critical threshold, however, it is still possible to obtain local boundedness without assuming higher integrability a priori; this is precisely the content of Theorem \ref{th2}. 

Let us start with the proof of Theorem~\ref{th1}. We treat separately the cases $\tx{r} \le \tx{m}$ and $\tx{r} > \tx{m} \ \& \ q \geq q_c$. In the first case, the argument follows a strategy similar to the local approach in \cite[Theorem 1.5]{BDGLS}, while in the second case we adapt techniques inspired by \cite[Theorem 1.2]{C} to the nonlocal setting.

\begin{proof}[Proof of Theorem \ref{th1}]
    Let $u$ be a local, nonnegative, weak sub-solution of \eqref{pm} in $ \Omega_T$ and consider $\tx{r}\geq q+1$.  Let  $Q_{R,\sigma}(x_0,t_0)\Subset \Omega_T$, for $0<\sigma$, $0<\rr<R$. Assume for simplicity $(x_0,t_0)=(0,0)$. For $\lambda \in (0,1)$, some $k>0$ to be chosen and $n \in \mathbb{N}_0$, consider 
    \eqn{sis1}
$$ \begin{cases}
     k_n = k - \frac{k}{2^{n+1}}, \qquad &\tilde{k}_n = \frac{k_n + k_{n+1}}{2},
    \\  \rr_n = \lambda \rr + \frac{(1-\lambda)\rr}{2^{n+1}}, \qquad &\tilde{\rr}_n = \frac{\rr_n + \rr_{n+1}}{2}, \\ \sigma_n = \lambda \sigma + \frac{(1-\lambda)\sigma}{2^{n+1}}, \qquad &\tilde{\sigma}_n = \frac{\sigma_n + \sigma_{n+1}}{2}, \\ B_n= B_{\rr_n}, \qquad & \Tilde{B}_n = B_{\tilde{\rr}_n}, \\ Q_n = B_n \times (-\sigma_n, 0), \qquad & \tilde{Q}_n=\tilde{B}_n \times (-\tilde{\sigma}_n,0).
\end{cases}$$
Define also, for $n \in \mathbb{N}_0$, $t \in (-\sigma,0)$,
$$ {k}_n^t:= k_n(t):=k_n + \ell(t), \qquad \tilde{k}_n^t :=\tilde{k}_n(t):= \frac{{k}_n^t + {k}_{n+1}^t}{2},$$
where
$$\ell(t):=\left[2\tx{L}\left(\frac{R}{R-{\rr}}\right)^{N+2s}\int_{-\sigma}^t\int_{\R^N\setminus B_R}\frac{u(x,\tau)}{|x|^{N+2s}}\dx\mathrm{
d
}\tau\right]^{\frac1q}.$$
Noting that $\tilde{k}_n^t= \tilde{k}_n + \ell(t)$, $0<\tsn<\sn<\sigma$ and $0<\tilde{\rr}_n<\rr_n<\rr<R$, we use the Caccioppoli estimate \eqref{caccioppoli} in this setting, obtaining
\begin{align} \label{caccN}
&\sup_{t\in(-\tsn,0)}\int_{\tilde{B}_n}\mathfrak{g}_{+}(u(x,t),\tilde{k}_n(t))\dx
+\int_{-\tsn}^0\!\!\iint_{\tilde{B}_n\times \tilde{B}_n}\frac{|(u(x,t)-\tilde{k}_n(t))_{+}-(u(y,t)-\tilde{k}_n(t))_{+}|^2}{|x-y|^{N+2s}}\dx\dy\dt \notag\\ 
& \quad \le\frac{c}{\sn-\tsn}\iint_{Q_n}\mathfrak{g}_{+}(u(x,t),\tilde{k}_n(t))\dx\dt
+c\frac{\rr_n^{2(1-s)}}{(\rr_n-{\tilde{\rr}_n})^2}\iint_{Q_n}(u(x,t)-\tilde{k}_n(t))_{+}^2\dx\dt \notag\\
&\qquad+ c R^N \left ( \frac{1}{\rr_n-\tilde{\rr}_n}\right )^{N+2s}\sup_{t \in [-\sigma,0]} \mint_{B_R} u(y,t) \dy\iint_{Q_{n}}  (u(x,t)-\tilde{k}_n(t))_{+} \dx  \dt \notag \\ & \quad \leq c \frac{2^{n}}{(1-\lambda)\sigma} \iint_{Q_n}\mathfrak{g}_{+}(u(x,t),\tilde{k}_n(t))\dx\dt+
c\frac{2^{2n}}{(1-\lambda)^{2}\rr^{2s}}\iint_{Q_n}(u(x,t)-\tilde{k}_n(t))_{+}^2\dx\dt \notag\\
&\qquad+c\left (\frac{R}{\rr} \right )^N\frac{2^{n(N+2s)}}{(1-\lambda)^{N+2s} \rr^{2s}}\sup_{t \in [-\sigma,0]} \mint_{B_R} u(y,t) \dy\iint_{Q_{n}}  (u(x,t)-\tilde{k}_n(t))_{+} \dx  \dt,
\end{align}
with $c\equiv c(s,N,\tx{L})$. Now, on the set $\{u>k_{n+1}^t \}$, by Lemma \ref{2.3}, it holds 
\begin{align} \label{g1}
\mathfrak{g}_+(u,\tilde{k}_n^t) \gtrsim_q (u+\tilde{k}_n^t)^{q-1}(u-\tilde{k}_n^t)_+^2  \gtrsim_q (u-\tilde{k}_n^t)^{q-1}(u-\tilde{k}_n^t)_+^2 \gtrsim_q (u-k_{n+1}^t)^{q+1}_+.
\end{align}
Take,
\eqn{0k}
$$ k \geq \sup_{t \in (-\sigma,0)}\ell(t)= (2\tx{L})^{\frac{1}{q}}\left(\frac{R}{R-{\rr}}\right)^{\frac{N+2s}{q}} \textnormal{Tail}^{\frac{1}{q}}(u; Q_{R,\sigma}).$$
On the set $\{ u > \tilde{k}_n^t\}$, using \eqref{0k}, we have
$$ 1 \leq \frac{u+\tilde{k}_n^t}{u-{k}_n^t} \leq \frac{2u}{u-{k}_n^t} \leq 2\frac{\tilde{k}_n + \ell}{\tilde{k}_n-{k}_n} \lesssim 2^{n}.$$
Thereby, using again Lemma \ref{2.3}, it holds
\begin{align} \label{g2}
\mathfrak{g}_+(u, \tilde{k}_n^t) 
&\lesssim_q (u + \tilde{k}_n^t)^{q-1} (u - \tilde{k}_n^t)_+^2 \notag \\
&\lesssim_q 2^{n(q-1)} (u - k_n^t)_+^{q-1} (u - \tilde{k}_n^t)_+^2 \notag  \\
&\lesssim_q 2^{n(q-1)} (u - k_n^t)_+^{q+1} \chi_{\{u > \tilde{k}_n^t\}}.
\end{align}
Setting $\tilde{A}_n:= \{ u >\tilde{k}_n^t\} \cap Q_n$, and using \eqref{g1}, \eqref{g2} in \eqref{caccN}, we get
\begin{align} \label{cacc3}
&\sup_{t\in(-\tsn,0)}\int_{\tilde{B}_n}(u-k^t_{n+1})^{q+1}_+\dx
+\int_{-\tsn}^0\!\!\iint_{\tilde{B}_n\times \tilde{B}_n}\frac{|(u(x,t)-\tilde{k}_n^t)_{+}-(u(y,t)-\tilde{k}_n^t)_{+}|^2}{|x-y|^{N+2s}}\dx\dy\dt \notag\\ 
& \quad \le c\frac{2^{nq}}{(1-\lambda)\sigma} \iint_{\tilde{A}_n}(u-k^t_n)_{+}^{q+1}\dx\dt
+c\frac{2^{2n}}{(1-\lambda)^2 \rr^{2s}}\iint_{Q_n}(u-\tilde{k}_n^t)_{+}^2\dx\dt \notag\\
&\qquad+c\left (\frac{R}{\rr} \right )^N\frac{2^{n(N+2s)}}{(1-\lambda)^{N+2s} \rr^{2s}}\sup_{t \in [-\sigma,0]} \mint_{B_R} u \dy\iint_{Q_{n}}  (u-\tilde{k}_n^t)_{+} \dx  \dt,
\end{align}
with $c\equiv c(s,q,N,\tx{L})$. Using the measure estimate
\eqn{at}
$$ |\tilde{A}_n|=|\{u>\tilde{k}_n^t\} \cap Q_n| \leq \frac{2^{(n+3)\tx{r}}}{k^{\tx{r}}} \iint_{Q_n}(u-k_n^t)^\tx{r}_+ \dx \dt,$$
and H\"older inequality, we have
\eqn{h1}
$$ \iint_{\tilde{A}_n}(u-k_n^t)_+^{q+1} \dx \dt \leq \left ( \iint_{Q_n} (u-k_n^t)^\tx{r}_+ \dx\dt\right )^{\frac{q+1}{\tx{r}}} |\tilde{A}_n|^{1-\frac{q+1}{\tx{r}}} \lesssim_q \frac{2^{(\tx{r}-q-1)n}}{k^{\tx{r}-q-1}} \iint_{Q_n}(u-k_n^t)^{\tx{r}}_+ \dx \dt,$$
\eqn{h2}
$$ \iint_{Q_n}(u-\tilde{k}_n^t)^2_+ \dx \dt \leq \left ( \iint_{Q_n}(u-\tilde{k}_n^t)^\tx{r}_+ \dx \dt\right )^{\frac{2}{\tx{r}}}|\tilde{A}_n|^{1-\frac{2}{\tx{r}}} \lesssim_q \frac{2^{(\tx{r}-2)n}}{k^{\tx{r}-2}}\iint_{Q_n}(u-k_n^t)_+^{\tx{r}} \dx \dt$$
and
\eqn{h3}
$$ \iint_{Q_n}(u-\tilde{k}_n^t)_+ \dx \dt \leq \left ( \iint_{Q_n}(u-\tilde{k}_n^t)^\tx{r}_+ \dx \dt\right )^{\frac{1}{\tx{r}}}|\tilde{A}_n|^{1-\frac{1}{\tx{r}}} \lesssim_q \frac{2^{(\tx{r}-1)n}}{k^{\tx{r}-1}}\iint_{Q_n}(u-k_n^t)_+^{\tx{r}} \dx \dt.$$
Then, using \eqref{h1}-\eqref{h3} in \eqref{cacc3}, we get
\begin{align} \label{cacc4}
&\sup_{t\in(-\tsn,0)}\int_{\tilde{B}_n}(u-k^t_{n+1})^{q+1}_+\dx
+\int_{-\tsn}^0\!\!\iint_{\tilde{B}_n\times \tilde{B}_n}\frac{|(u(x,t)-{k}_{n+1}^t)_{+}-(u(y,t)-{k}_{n+1}^t)_{+}|^2}{|x-y|^{N+2s}}\dx\dy\dt \notag\\ 
& \quad \le \frac{c}{k^{\tx{r}-q-1}}\frac{2^{(\tx{r}-1)n}}{(1-\lambda)\sigma} \iint_{Q_n}(u-k^t_n)_{+}^{\tx{r}}\dx\dt
+\frac{c}{k^{\tx{r}-2}}\frac{2^{\tx{r}n}}{(1-\lambda)^2 \rr^{2s}}\iint_{Q_n}(u-{k}_n^t)_{+}^\tx{r}\dx\dt \notag \\
&\qquad+\frac{c}{k^{\tx{r}-1}}\left (\frac{R}{\rr} \right )^N\frac{2^{n(N+2s+\tx{r}-1)}}{(1-\lambda)^{N+2s} \rr^{2s}}\sup_{t \in [-\sigma,0]} \mint_{B_R} u \dy\iint_{Q_{n}}  (u-k_n^t)^\tx{r}_{+} \dx  \dt \notag \\ &  \quad \leq \frac{c}{k^{\tx{r}-1-q}}\frac{2^{n(N+2s +\tx{r}-1)}}{(1-\lambda)^{N+2s} \sigma} \left (1+\frac{1}{k^{q-1}}\frac{\sigma}{\rr^{2s}}+ \frac{1}{k^q}\left(\frac{R}{\rr} \right )^N \frac{\sigma}{ \rr^{2s}}\sup_{t \in [-\sigma,0]} \mint_{B_R} u \dy\right ) \iint_{Q_n} (u-k^t_n)^\tx{r}_+ \dx \dt,
\end{align}
Note that we used the fact that  since $\tilde k_n^t\le k_{n+1}^t$, we have 
\[
|
(u(x,t)-k_{n+1}^t)_+
-
(u(y,t)-k_{n+1}^t)_+
|
\le
|
(u(x,t)-\tilde k_n^t)_+
-
(u(y,t)-\tilde k_n^t)_+
|.
\]
Now,  we choose
\eqn{k}
$$ k \geq \left ( \frac{\sigma}{\rr^{2s}} \right )^{\frac{1}{q-1}}, \qquad k \geq \left ( \delta \frac{\sigma}{\rr^{2s}} \left(\frac{R}{\rr} \right )^N\sup_{t \in [-\sigma,0]} \mint_{B_R} u(y,t) \dy\right )^{\frac{1}{q}},$$
for $\delta>0$ that will be fixed later. Therefore, from \eqref{cacc4} using \eqref{k}, we get
\begin{align} \label{cacc5}
&\sup_{t\in(-\tsn,0)}\int_{\tilde{B}_n}(u-k^t_{n+1})^{q+1}_+\dx
+\int_{-\tsn}^0\!\!\iint_{\tilde{B}_n\times \tilde{B}_n}\frac{|(u(x,t)-{k}_{n+1}^t)_{+}-(u(y,t)-{k}_{n+1}^t)_{+}|^2}{|x-y|^{N+2s}}\dx\dy\dt \notag\\ 
& \quad \le \frac{c}{k^{\tx{r}-q-1}}\frac{2^{n(N+2s+\tx{r}-1)}}{(1-\lambda)^{N+2s}\sigma} \left ( 1+\frac{1}{\delta} \right ) \iint_{Q_n} (u-k_n^t)_+^\tx{r}
\dx \dt, 
\end{align}
with $c \equiv c(s,q,N,\tx{L})$. Consider now a cut-off function $0\leq \phi\leq 1$  on $\tilde{Q}_n$, which vanishes outside
$$ B_{\rr_n/4 + 3\rr_{n+1}/4} \times \left ( -\frac{1}{4}\sigma_n -\frac{3}{4}\sigma_{n+1},0\right )$$
and 
$\phi_{|_{Q_{n+1}}}=1$, $|D\phi| \leq 2^{n+4}/\rr$. Recalling the definition of $\tx{m}$ in \eqref{m}, we apply the parabolic Sobolev embedding Lemma \ref{em} with $d=(1-\lambda)2^{-n-6}$, obtaining
\begin{align} \label{cc1}
  &  \iint_{Q_{n+1}} (u-k_{n+1}^t)_+^{\tx{m}} \dx \dt \notag \\ & \qquad \leq \iint_{\tilde{Q}_n} [(u-k_{n+1}^t)_+\phi]^{{\tx{m}}} \dx \dt \notag \\ & \qquad \leq c\left ( \rr^{2s} \int_{-\tsn}^0\!\!\iint_{\tilde{B}_n\times \tilde{B}_n}\frac{|\phi(x,t)(u(x,t)-{k}_{n+1}^t)_{+}-\phi(y,t)(u(y,t)-{k}_{n+1}^t)_{+}|^2}{|x-y|^{N+2s}}\dx\dy\dt\right . \notag  \\ & \qquad \qquad\left . +\frac{2^{n(N+2s)}}{(1-\lambda)^{N+2s}} \iint_{\tilde{Q}_n} [\phi(u - k_{n+1}^t)_+]^2 \dx \dt\right ) \left ( \sup_{t \in (-\tilde{\sigma}_n,0)} \mint_{\tilde{B}_n} [\phi(u-k_{n+1}^t)_+]^{q+1} \dx\right )^{\frac{2s}{N}} \notag \\ & \qquad =:c (\tx{R}_1 + \tx{R}_2) (\tx{R}_3)^{\frac{2s}{N}},
\end{align}
for $c \equiv c(s,q,N)$. Let us estimate the three terms in the right-hand side above. As regards $\tx{R}_1$, using
\begin{align*}   
&|(u(x,t)-k_{n+1}^t)_+ \phi(x, t) - (u(y,t)-k_{n+1}^t)_+ \phi(y, t)|^2 
\\ & \qquad \leq 2|(u(x,t)-k_{n+1}^t)_+ - (u(y,t)-k_{n+1}^t)_+|^2 \phi^2(x, t) 
+ 2(u(y,t)-k_{n+1}^t)_+^2|\phi(x, t) - \phi(y, t)|^2,
\end{align*}
and \eqref{h2}, \eqref{cacc5}, $(\ref{k})_1$ we get
\begin{align*}
    \tx{R}_1 & \lesssim  \rr^{2s} \int_{-\tilde{\sigma}_n}^0 \iint_{\tilde{B}_n \times \tilde{B}_n} \frac{|(u(x,t)-{k}_{n+1}^t)_{+}-(u(y,t)-{k}_{n+1}^t)_{+}|^2}{|x-y|^{N+2s}}\dx\dy \dt \notag \\ & \qquad +  2^{2n}\iint_{\tilde{Q}_n}(u-k_{n+1}^t)_+^2 \dx \dt \notag \\ &  \lesssim_{s,q,N,\tx{L}} \frac{\rr^{2s}}{k^{\tx{r}-q-1}}\frac{2^{n(N+2s+\tx{r}-1)}}{(1-\lambda)^{N+2s}\sigma} \left ( 1+\frac{1}{\delta}\right )\iint_{Q_n} (u-k_n^t)_+^{\tx{r}} \dx \dt. 
\end{align*}
For $\tx{R}_2$, using \eqref{h2}, it holds
$$ \tx{R}_2 \lesssim_{s,q,N} \frac{1}{k^{\tx{r}-2}}\frac{2^{n(N+2s+\tx{r}-2)}}{(1-\lambda)^{N+2s}} \iint_{{Q}_n}(u-k^t_{n})_+^{\tx{r}} \dx \dt, $$
therefore, by $(\ref{k})_1$,
$$ \tx{R}_1 + \tx{R}_2 \lesssim_{s,q,N,\tx{L}}  \frac{\rr^{2s}}{k^{\tx{r}-q-1}}\frac{2^{n(N+2s+\tx{r}-1)}}{(1-\lambda)^{N+2s}\sigma} \left ( 1+\frac{1}{\delta}\right )\iint_{Q_n} (u-k_n^t)_+^{\tx{r}} \dx \dt.$$
Finally, using again the energy estimate \eqref{cacc5}, we obtain for $\tx{R}_3$ 
\[
\tx{R}_3
\lesssim_{s,q,N,L}
\frac{1}{\varrho^N}
\frac{1}{k^{\tx{r}-q-1}}
\frac{2^{n(N+2s+\tx{r}-1)}}{(1-\lambda)^{N+2s}\sigma}
\left(1+\frac1\delta\right)
\iint_{Q_n}(u-k_n^t)_+^\tx{r}\dx\dt.
\]
Therefore, denoting with $\alpha(n):= n(N+2s+\tx{r}-1)(N+2s)/N$ and collecting the bounds for $\tx{R}_1, \tx{R}_2,\tx{R}_3$, we continue estimating \eqref{cc1} as
\begin{align} \label{stima}
 &   \iint_{Q_{n+1}}(u-k_{n+1}^t)_+^{\tx{m}} \dx \dt \notag \\ & \qquad \leq \frac{c}{k^{\frac{(\tx{r}-q-1)(N+2s)}{N}}} \frac{2^{\alpha(n)}}{(1-\lambda)^{\frac{(N+2s)^2}{N}}\sigma^{\frac{N+2s}{N}}} \left (1+\frac{1}{\delta} \right )^{\frac{N+2s}{N}} \left (\iint_{Q_n}(u-k_n^t)^{\tx{r}}_+ \dx \dt \right )^{\frac{N+2s}{N}}
\end{align}
that holds whenever $\tx{r} \geq q+1$, with $c\equiv c(s,q,N,\tx{L})$. Now, we distinguish three cases according to the relation between $\tx{r}$ and $q+1$, $\tx{m}$. \\ \\ 
\textbf{Case 1}: $q+1 \leq \tx{r} \leq \tx{m}$, $\Lambda_{\tx{r}} >0$.\\ 
\noindent
Let 
$$ Y_n := \iint_{Q_n} (u-k_n^t)_+^{\tx{r}} \dx \dt,$$
and observe that by H\"older inequality and \eqref{at}, it follows that
\begin{align} \label{ss}
     Y_{n+1} & \leq \left ( \iint_{Q_{n+1}} (u-k_{n+1}^t)_+^{\tx{m}} \dx \dt \right )^{\frac{\tx{r}}{\tx{m}}}|\tilde{A}_n|^{1-\frac{\tx{r}} {\tx{m}}} \notag \\ &  \leq c\left ( \iint_{Q_{n+1}} (u-k_{n+1}^t)_+^{\tx{m}} \dx \dt \right )^{\frac{\tx{r}}{\tx{m}}} \left ( \frac{2^{\tx{r}n}}{k^\tx{r}} \iint_{Q_n}(u-k_n^t)^\tx{r}_+ \dx \dt \right )^{1-\frac{\tx{r}}{\tx{m}}}.
\end{align} 
Therefore, by \eqref{stima} and \eqref{ss}, we get
$$ Y_{n+1} \leq \frac{c}{k^\frac{{\tx{r}\Lambda_{\tx{r}}}}{\tx{m}N}}\frac{2^{\beta(n)}}{(1-\lambda)^\frac{\tx{r}(N+2s)^2}{\tx{m}N}\sigma^{\frac{(N+2s)\tx{r}}{\tx{m}N}}}  Y_n^{1+\frac{2s\tx{r}}{\tx{m}N}}\left ( 1+\frac{1}{\delta}\right )^{\frac{\tx{r}(N+2s)}{\tx{m}N}},$$
with $\beta(n):= \alpha(n)\tx{r}/\tx{m} +n\tx{r}(\tx{m}-\tx{r})/\tx{m}$ and
$ \Lambda_{\tx{r}}$ defined accordingly with \eqref{Lambda}.
Now, we apply the fast geometric convergence Lemma \ref{fg} obtaining that there exists a positive constant $c$, depending only on $s,q,N,\tx{L}$, such that $Y_n \to 0$ if
$$ Y_0 = \iint_{Q_0} \left(u(x,t)- \left(\ell(t)+\frac{k}{2}\right)\right)_+^{\tx{r}}\dx \dt \leq c^{-1} (1-\lambda)^{\frac{(N+2s)^2}{2s}} \sigma^{\frac{N+2s}{2s}}k^{\frac{\Lambda_{\tx{r}}}{2s}} \left (1+\frac{1}{\delta} \right )^{-\frac{N+2s}{2s}},$$
where the resulting power of $2$ has been absorbed into the constant $c$. The inequality above is satisfied if
\eqn{k1}
$$ k \geq \frac{c}{(1-\lambda)^{\frac{(N+2s)^2}{\Lambda_{\tx{r}}}}\sigma^{\frac{N+2s}{\Lambda_{\tx{r}}}}}\left ( 1+\frac{1}{\delta}\right )^{\frac{N+2s}{\Lambda_{\tx{r}}}}\left ( \iint_{Q_0} u^\tx{r}(x,t) \dx \dt \right )^{\frac{2s}{\Lambda_{\tx{r}}}}.$$
Taking \eqref{0k},\eqref{k} and \eqref{k1} into account, we obtain
\begin{align} \label{estimatedelta}
    \operatorname*{ess\,sup}_{Q_{\lambda \rr, \lambda \sigma}} u& \leq \frac{c}{(1-\lambda)^{\frac{(N+2s)^2}{\Lambda_{\tx{r}}}}\sigma^{\frac{N+2s}{\Lambda_{\tx{r}}}}} \left ( 1+\frac{1}{\delta}\right )^{\frac{N+2s}{\Lambda_{\tx{r}}}}\left ( \iint_{Q_{\rr,\sigma}} u^{\tx{r}}(x,t) \dx \dt \right )^{\frac{2s}{\Lambda_{\tx{r}}}} \notag \\ & \qquad + \left ( \delta \frac{\sigma}{\rr^{2s}} \left ( \frac{R}{\rr} \right )^N \operatorname*{ess\,sup}_{t \in [-\sigma,0]} \mint_{B_R} u(x,t) \dx \right )^{\frac{1}{q}} + \left( \frac{\sigma}{\rr^{2s}}\right)^{\frac{1}{q-1}} \notag\\ & \qquad +c \left ( \frac{R}{R-\rr}\right )^{\frac{N+2s}{q}}
    \textnormal{Tail}^{\frac{1}{q}}(u;Q_{R,\sigma}),
\end{align}
that holds for all $Q_{R,\sigma} \Subset \Omega_T$, $\lambda \in (0,1)$, $\delta>0$ and $\rr \in (0,R)$; with $c\equiv c(s,q,N,\tx{L})$. Now, we want to refine the previous estimate using an interpolation argument. Introduce for $n \in \mathbb{N}_0$
\eqn{ia}
$$
\begin{cases}
     \rr_n = \rr - \frac{(1-\lambda)\rr}{2^n}, \quad & \tilde{\rr}_n = \frac{\rr_n + \rr_{n+1}}{2} \\ \sigma_n = \sigma - \frac{(1-\lambda)\sigma}{2^n}, \quad & \tilde{\sigma}_n = \frac{\sigma_n + \sigma_{n+1}}{2} \\ \lambda_n = \frac{\rr_n}{\tilde{\rr}_n}= \frac{\sigma_n}{\tilde{\sigma}_n}.
\end{cases}
$$
Applying \eqref{estimatedelta} with $(\lambda_n, \tilde{\rr}_n, \rr_{n+1}, \sigma_{n+1})$ in place of $(\lambda, \rr, R, \sigma)$ and noticing that $\sigma_n \leq \lambda_n \sigma_{n+1}$, we get
\begin{align} \label{estimatedelta2}
    \operatorname*{ess\,sup}_{Q_{\rr_n, \sigma_n}} u & \leq c \frac{\tilde{\rr}_n^{\frac{2sN}{\Lambda_{\tx{r}}}}}{(1-\lambda_n)^{\frac{(N+2s)^2}{\Lambda_{\tx{r}}}}\sigma_{n+1}^{\frac{N}{\Lambda_{\tx{r}}}}} \left ( 1+\frac{1}{\delta}\right )^{\frac{N+2s}{\Lambda_{\tx{r}}}}\left ( \mint\mint_{Q_{\tilde{\rr}_n,\sigma_{n+1}}} u^{\tx{r}}(x,t) \dx \dt \right )^{\frac{2s}{\Lambda_{\tx{r}}}} \notag \\ & \qquad + \left ( \delta \frac{\sigma_{n+1}}{\tilde{\rr}_n^{2s}} \left ( \frac{\rr_{n+1}}{\tilde{\rr}_n} \right )^N \operatorname*{ess\,sup}_{Q_{\rr_{n+1}, \sigma_{n+1}}} u\right )^{\frac{1}{q}} + \left( \frac{\sigma_{n+1}}{\tilde{\rr}_n^{2s}}\right)^{\frac{1}{q-1}} \notag\\ & \qquad +c \left ( \frac{\rr_{n+1}}{\rr_{n+1}-\tilde{\rr}_n}\right )^{\frac{N+2s}{q}} \text{Tail}^{\frac{1}{q}}(u; Q_{\rr_{n+1}, \sigma_{n+1}}) := \tilde{\tx{R}}_1 + \tilde{\tx{R}}_2 + \tilde{\tx{R}}_3 + \tilde{\tx{R}}_4.
\end{align}
Let us estimate the four terms in the right-hand side above. For $\tilde{\tx{R}}_1$, we observe that
\[
\begin{aligned}
(1 - \lambda_n)^{N+1}\tilde{\varrho}_n^N \sigma_{n+1} 
&= (\tilde{\varrho}_n - \varrho_n)^N (1 - \lambda_n) \sigma_{n+1} \\
&= (\varrho_{n+1} - \varrho_n)^N (\tilde{\sigma}_n - \sigma_n) \frac{\sigma_{n+1}}{\tilde{\sigma}_n} \\
&\geq \frac{1}{2} (\varrho_{n+1} - \varrho_n)^N (\sigma_{n+1} - \sigma_n) \\
&= \frac{\varrho^N \sigma (1 - \lambda)^{N+1}}{2^{1+(n+1)(N+1)}}
\end{aligned}
\]
and 
$$1 - \lambda_n = 1 - \frac{\varrho_n}{\tilde{\varrho}_n} = \frac{\varrho_{n+1} - \varrho_n}{\tilde{\varrho}_n} \geq \frac{1 - \lambda}{2^{n+2}}, $$
then,
\begin{align} \label{delta1}
    \tilde{\tx{R}}_1 & = \frac{c}{(1-\lambda_n)^{\frac{(N+2s)^2-2s(N+1)}{\Lambda_{\tx{r}}}}} \left ( \frac{\tilde{\rr}_n^{2s}}{\sigma_{n+1}} \right )^{\frac{N}{\Lambda_{\tx{r}}}}\left ( 1+\frac{1}{\delta}\right )^{\frac{N+2s}{\Lambda_{\tx{r}}}} \notag 
    \\ & \qquad \qquad \qquad \cdot \left ( \frac{1}{(1-\lambda_n)^{N+1}\tilde{\rr}_n^N \sigma_{n+1}} \iint_{Q_{\tilde{\rr}_n,\sigma_{n+1}}} u^\tx{r}(x,t) \dx \dt\right )^{\frac{2s}{\Lambda_{\tx{r}}}} \notag \\ & \leq c \left ( \frac{2^n}{1-\lambda} \right )^{\frac{(N+2s)^{2}}{\Lambda_{\tx{r}}}}\left ( \frac{{\rr}^{2s}}{\sigma} \right )^{\frac{N}{\Lambda_{\tx{r}}}}\left ( 1+\frac{1}{\delta}\right )^{\frac{N+2s}{\Lambda_{\tx{r}}}} \left ( \mint\mint_{Q_{{\rr},\sigma}} u^\tx{r}(x,t) \dx \dt\right )^{\frac{2s}{\Lambda_{\tx{r}}}},
\end{align}
with $c \equiv c(s,q,N, \tx{L})$. As regards $\tilde{\tx{R}}_2$, we apply Young inequality with conjugate exponents $(q,q/(q-1))$, obtaining
\begin{align} \label{delta2}
     \tilde{\tx{R}}_2 & \leq c \left ( \delta \left ( \frac{\sigma_{n+1}}{\tilde{\rr}_n^{2s}} \right ) \operatorname*{ess\,sup}_{Q_{\rr_{n+1}, \sigma_{n+1}}} u \right )^{\frac{1}{q}} \notag \\ & \leq c \left (\left ( \delta \operatorname*{ess\,sup}_{Q_{\rr_{n+1},\sigma_{n+1}}} u \right )^q +  \left ( \frac{\sigma_{n+1}}{\tilde{\rr}_n^{2s}} \right )^{\frac{q}{q-1}} \right )^{\frac{1}{q}} \notag \\ & \leq c \delta \operatorname*{ess\,sup}_{Q_{\rr_{n+1},\sigma_{n+1}}} u +c \left ( \frac{\sigma}{\rr^{2s}}\right )^{\frac{1}{q-1}},
\end{align}
with $c\equiv c(q,N)$. For $\tilde{\tx{R}}_3, \tilde{\tx{R}}_4$, we easily observe that
\eqn{delta3}
$$ \tilde{\tx{R}}_3 \lesssim_q \left (\frac{\sigma}{\rr^{2s}} \right )^{\frac{1}{q-1}} $$
\eqn{delta4}
$$ \tilde{\tx{R}}_4 \lesssim_{s,q,N,\tx{L}}  \left ( \frac{2^{n}}{1-\lambda} \right )^{\frac{N+2s}{q}} \text{Tail}^{\frac{1}{q}}(u; Q_{\lambda \rr, \sigma}).$$
Putting together \eqref{delta1}-\eqref{delta4}, we continue estimating \eqref{estimatedelta2} as
\begin{align} \label{essdelta}
    \operatorname*{ess\,sup}_{Q_{\rr_n, \sigma_n}} u & \leq c \left ( \frac{2^n}{1-\lambda} \right )^{\frac{(N+2s)^{2}}{\Lambda_{\tx{r}}}}\left ( \frac{{\rr}^{2s}}{\sigma} \right )^{\frac{N}{\Lambda_{\tx{r}}}}\left ( 1+\frac{1}{\delta}\right )^{\frac{N+2s}{\Lambda_{\tx{r}}}} \left ( \mint\mint_{Q_{{\rr},\sigma}} u^\tx{r}(x,t) \dx \dt\right )^{\frac{2s}{\Lambda_{\tx{r}}}} \notag \\ & \qquad + c\delta \operatorname*{ess\,sup}_{Q_{\rr_{n+1},\sigma_{n+1}}} u + c\left ( \frac{\sigma}{\rr^{2s}}\right )^{\frac{1}{q-1}}\notag  \\ & \qquad + c \left ( \frac{2^{n}}{1-\lambda} \right )^{\frac{N+2s}{q}} \text{Tail}^{\frac{1}{q}}(u; Q_{\lambda \rr, \sigma}),
\end{align}
with $c \equiv c(s,q,N,\tx{L})$. Denote
$$ \Bar{\delta}:= c \delta, \quad \tx{b}:= 2^{\frac{(N+2s)^2}{\Lambda_{\tx{r}}} +\frac{N+2s}{q}+1}$$
\begin{align*}
    \tx{B}_\delta & :=  c \left ( \frac{1}{1-\lambda} \right )^{\frac{(N+2s)^{2}}{\Lambda_{\tx{r}}}}\left ( \frac{{\rr}^{2s}}{\sigma} \right )^{\frac{N}{\Lambda_{\tx{r}}}}\left ( 1+\frac{1}{\delta}\right )^{\frac{N+2s}{\Lambda_{\tx{r}}}} \left ( \mint\mint_{Q_{{\rr},\sigma}} u^\tx{r}(x,t) \dx \dt\right )^{\frac{2s}{\Lambda_{\tx{r}}}} \\ & \qquad + c\left ( \frac{\sigma}{\rr^{2s}}\right )^{\frac{1}{q-1}} +  c \left ( \frac{1}{1-\lambda} \right )^{\frac{N+2s}{q}} \text{Tail}^{\frac{1}{q}}(u; Q_{\lambda \rr, \sigma}),
\end{align*}
therefore \eqref{essdelta} reads as
$$ \operatorname*{ess\,sup}_{Q_{\rr_n, \sigma_n}} u \leq \Bar{\delta}\operatorname*{ess\,sup}_{Q_{\rr_{n+1},\sigma_{n+1}}} u + \tx{B}_\delta \tx{b}^n. $$
Iterating the recursive inequality above we get
$$ \operatorname*{ess\,sup}_{Q_{\rr_0,\sigma_0}} u \leq \Bar{\delta}^n \operatorname*{ess\,sup}_{Q_{\rr_{n},\sigma_{n}}} u + \tx{B}_\delta \sum_{i=0}^{n-1} (\Bar{\delta}\tx{b})^i,$$
so choosing $\Bar{\delta}= 1/(2\tx{b})$ and letting $n \to \infty$, we obtain
\begin{align} \label{fe}
\operatorname*{ess\,sup}_{Q_{\lambda \rr, \lambda \sigma}} u & \leq c \left ( \frac{1}{1-\lambda} \right )^{\frac{(N+2s)^{2}}{\Lambda_{\tx{r}}}}\left ( \frac{{\rr}^{2s}}{\sigma} \right )^{\frac{N}{\Lambda_{\tx{r}}}}\left ( \mint\mint_{Q_{{\rr},\sigma}} u^\tx{r}(x,t) \dx \dt\right )^{\frac{2s}{\Lambda_{\tx{r}}}} \notag \\ & \qquad + c\left ( \frac{\sigma}{\rr^{2s}}\right )^{\frac{1}{q-1}} +  c \left ( \frac{1}{1-\lambda} \right )^{\frac{N+2s}{q}} \text{Tail}^{\frac{1}{q}}(u; Q_{\lambda \rr, \sigma}),
\end{align}
that with the choice of $\lambda=1/2$ yields \eqref{tesi1}; for some $c\equiv c(s,q,N,\tx{L})$. This concludes case 1.
 \\ \\
\textbf{Case 2}: $\tx{r} \leq \tx{m}$ and $\tx{r}<q+1$, $\Lambda_\tx{r}>0$. \\ \noindent Note that, since $\Lambda_{q+1} > \Lambda_{\tx{r}}>0$ it holds $q+1 < \tx{m}$. Then, we can apply the results from the previous case with $q+1$ in place of $\tx{r}$. Hence, $u \in L^{\infty}_{\loc}(\Omega_T)$ and \eqref{fe} reads as
\begin{align} \label{fe1}
\operatorname*{ess\,sup}_{Q_{\lambda \rr, \lambda \sigma}} u & \leq c \left ( \frac{1}{1-\lambda} \right )^{\frac{(N+2s)^{2}}{\Lambda_{q+1}}}\left ( \frac{{\rr}^{2s}}{\sigma} \right )^{\frac{N}{\Lambda_{q+1}}}\left ( \mint\mint_{Q_{{\rr},\sigma}} u^{q+1}(x,t) \dx \dt\right )^{\frac{2s}{\Lambda_{q+1}}} \notag \\ & \qquad + c\left ( \frac{\sigma}{\rr^{2s}}\right )^{\frac{1}{q-1}} +  c \left ( \frac{1}{1-\lambda} \right )^{\frac{N+2s}{q}} \text{Tail}^{\frac{1}{q}}(u; Q_{\lambda \rr, \sigma}),
\end{align}
for any $\lambda \in (0,1)$, provided $Q_{\rr,\sigma} \Subset \Omega_T$. Defining
$$ \tx{M}:= \operatorname*{ess\,sup}_{Q_{\rr, \sigma}} u, \qquad \tx{M}_{\lambda}:= \operatorname*{ess\,sup}_{Q_{\lambda \rr, \lambda \sigma}} u,$$
from \eqref{fe1} we obtain
\begin{align} \label{fe2}
\tx{M}_\lambda & \leq c \tx{M}^{1-\frac{\Lambda_{\tx{r}}}{\Lambda_{q+1}}} \left ( \frac{1}{1-\lambda} \right )^{\frac{(N+2s)^{2}}{\Lambda_{q+1}}}\left ( \frac{{\rr}^{2s}}{\sigma} \right )^{\frac{N}{\Lambda_{q+1}}}\left ( \mint\mint_{Q_{{\rr},\sigma}} u^{\tx{r}}(x,t) \dx \dt\right )^{\frac{2s}{\Lambda_{q+1}}} \notag \\ & \qquad + c\left ( \frac{\sigma}{\rr^{2s}}\right )^{\frac{1}{q-1}} +  c \left ( \frac{1}{1-\lambda} \right )^{\frac{N+2s}{q}} \text{Tail}^{\frac{1}{q}}(u; Q_{\lambda \rr, \sigma}).
\end{align}
Applying Young inequality with conjugate exponents
$$\left ( \frac{\Lambda_{q+1}}{\Lambda_{\tx{r}}}, \frac{\Lambda_{q+1}}{\Lambda_{q+1}-\Lambda_{\tx{r}}} \right ),$$
to the first term in the right-hand side of \eqref{fe2}, we get
\begin{align} \label{fe3}
\tx{M}_\lambda & \leq \frac{1}{2} \tx{M} + c\left ( \frac{1}{1-\lambda} \right )^{\frac{(N+2s)^{2}}{\Lambda_{\tx{r}}}}\left ( \frac{{\rr}^{2s}}{\sigma} \right )^{\frac{N}{\Lambda_{\tx{r}}}}\left ( \mint\mint_{Q_{{\rr},\sigma}} u^{\tx{r}}(x,t) \dx \dt\right )^{\frac{2s}{\Lambda_{\tx{r}}}} \notag \\ & \qquad + c\left ( \frac{\sigma}{\rr^{2s}}\right )^{\frac{1}{q-1}} +  c \left ( \frac{1}{1-\lambda} \right )^{\frac{N+2s}{q}} \text{Tail}^{\frac{1}{q}}(u; Q_{\lambda \rr, \sigma}).
\end{align}
Now, we apply \eqref{fe3} to cylinders $Q_{\lambda_2\rr,\lambda_2 \sigma}$ and $Q_{\lambda_1\rr,\lambda_1 \sigma}$ with $1/2 < \lambda_1 < \lambda_2 < 1$ and we obtain
\begin{align*}
\tx{M}_{\lambda_1} & \leq \frac{1}{2} \tx{M}_{\lambda_2} + \left ( \frac{1}{\lambda_2-\lambda_1} \right )^{\frac{(N+2s)^{2}}{\Lambda_{\tx{r}}}}\left ( \frac{{\rr}^{2s}}{\sigma} \right )^{\frac{N}{\Lambda_{\tx{r}}}}\left ( \mint\mint_{Q_{{\lambda_2\rr},\lambda_2\sigma}} u^{\tx{r}}(x,t) \dx \dt\right )^{\frac{2s}{\Lambda_{\tx{r}}}} \notag \\ & \qquad + c\left ( \frac{\sigma}{\rr^{2s}}\right )^{\frac{1}{q-1}} +  c \left ( \frac{1}{\lambda_2-\lambda_1} \right )^{\frac{N+2s}{q}} \text{Tail}^{\frac{1}{q}}(u; Q_{\lambda_1 \rr, \lambda_2\sigma}).
\end{align*}
We conclude applying the iteration Lemma \ref{il} to $[1/2,1] \ni \lambda \mapsto \tx{M}_\lambda$, which yields \eqref{tesi1} and concludes the second case.
\\ \\ \textbf{Case 3}: $\tx{r} > \tx{m}$, $\Lambda_\tx{r} > 0$, $q \geq q_c$. \\ \noindent
In this case, we assume the higher integrability condition \eqref{hi}, so that $\tx{r}$ now denotes the exponent of higher integrability. Recalling the definition \eqref{m}, it holds $\tx{m}\leq q+1$, by $q \geq q_c$. \\ \\ \noindent Let us first analyze the case $\tx{m}<q+1$. Then from $(\ref{hi})_2$ and $q \geq q_c$ it follows that $\tx{r}>q+1$. We consider
\eqn{eps}
$$ (0,1)\ni \varepsilon := \frac{\tx{r}-1-q}{\tx{r}-\tx{m}}.$$ We observe that, by H\"older inequality with $(1/\varepsilon, 1/(1-\varepsilon))$ and the fact that $\tx{r}=(q+1-\varepsilon \tx{m})/(1-\varepsilon)$ by \eqref{eps}, we have 
\begin{align} \label{1}
    \iint_{Q_{n}}(u-k_{n}^t)_+^{q+1} \dx \dt & \leq \iint_{Q_{n}}(u-k_{n}^t)_+^{\varepsilon \tx{m}} u^{q+1-\varepsilon \tx{m}} \dx \dt \notag \\ & \leq \left ( \iint_{Q_{n}} (u-k_{n}^t)_+^{\tx{m}} \dx \dt \right )^\varepsilon \left ( \iint_{Q_{n}} u^\tx{r}(x,t) \dx \dt \right )^{1-\varepsilon} \notag \\ &  \leq \left ( \iint_{Q_{n}} (u-k_{n}^t)_+^{\tx{m}} \dx \dt\right )^\varepsilon \tx{D}^{1-\varepsilon},
\end{align}
where $$ \tx{D}:= \iint_{Q_{\rr,\sigma}} u^\tx{r}(x,t) \dx \dt.$$
Now, by iterative inequality \eqref{stima} with $\tx{r} = q+1$, we get
\begin{align} \label{2}
     \iint_{Q_{n+1}} (u-k_{n+1}^t)_+^{\tx{m}} \dx \dt \leq c \frac{2^{\alpha(n)}}{(1-\lambda)^{\frac{(N+2s)^2}{N}}\sigma^{\frac{N+2s}{N}}} \left (  1+ \frac{1}{\delta}\right )^{\frac{N+2s}{N}} \left ( \iint_{Q_n}(u-k_{n}^t)_+^{q+1} \dx \dt\right )^{\frac{N+2s}{N}}.
\end{align}
Then, setting
\eqn{Yn}
$$ Y_n := \iint_{Q_n}(u-k_{n}^t)_+^{\tx{m}} \dx \dt,$$
by \eqref{1} and \eqref{2}, we have
\begin{align*}
    Y_{n+1} \leq c \frac{2^{\alpha(n)}}{(1-\lambda)^{\frac{(N+2s)^2}{N}}\sigma^{\frac{N+2s}{N}}}\left (  1+ \frac{1}{\delta}\right )^{\frac{N+2s}{N}} Y_n^{\varepsilon \frac{N+2s}{N}} \tx{D}^{(1-\varepsilon)\frac{N+2s}{N}}.
\end{align*}
We observe that, by $(\ref{hi})_2$, it holds $\varepsilon \frac{N+2s}{N} >1,$ indeed
\begin{align*}
    \varepsilon (N+2s) -N &= \frac{(\tx{r}-1-q)(N+2s)-N(\tx{r}-\tx{m})}{\tx{r}-\tx{m}} \\ & = \frac{2s\tx{r} - (1+q)(N+2s) + 2N+2s(q+1)}{\tx{r}-\tx{m}} \\ & = \frac{2s\tx{r} +N(1-q)}{\tx{r}-\tx{m}} \overset{(\ref{hi})_2}{>}0.
\end{align*}
Then, we apply the fast geometric convergence Lemma \ref{fg} obtaining that there exists a positive constant $c$, depending only on $s,q,N,\tx{L},\tx{r}$, such that $Y_n \to 0$ if 
\eqn{Y0}
$$ Y_0  \leq c^{-1} \tx{D}^{-\frac{(N+2s)(1-\varepsilon)}{\varepsilon(N+2s)-N}} (1-\lambda)^{\frac{(N+2s)^2}{\varepsilon(N+2s)-N}}\sigma^{\frac{N+2s}{\varepsilon(N+2s)-N}}\left (  1+ \frac{1}{\delta}\right )^{-\frac{N+2s}{\varepsilon(N+2s)-N}}.$$
Now, let us observe that by H\"older inequality it holds
\begin{align} \label{Y1}
    Y_0 & = \iint_{Q_0}\left(u(x,t)-\frac{k}{2} - \ell(t)\right)_+^{\tx{m}} \dx \dt \notag \\ &\leq \iint_{Q_0}\left(u(x,t)-\frac{k}{2}\right)_+^{\tx{m}} \dx \dt \notag \\ &\leq\iint_{Q_0 \cap \{ u \geq k/2\}} u^{\tx{m}}(x,t) \dx \dt\notag \\ & \leq \left ( \iint_{Q_{\rr,\sigma}} u^\tx{r}(x,t) \dx \dt \right )^{\frac{\tx{m}}{\tx{r}}}|Q_0\cap \{ u \geq k/2\}|^{1-\frac{\tx{m}}{\tx{r}}} \notag \\ & \leq \tx{D} \frac{2^{\tx{r}-\tx{m}}}{k^{\tx{r}-\tx{m}}}.
\end{align}
Therefore, by \eqref{Y1}, taking
\eqn{k3}
$$ k \geq 2 \left [ 
\tx{D}^{1+ \frac{(N+2s)(1-\varepsilon)}{\varepsilon(N+2s)-N}} \frac{c}{(1-\lambda)^{\frac{(N+2s)^2}{\varepsilon(N+2s)-N}}\sigma^{\frac{N+2s}{\varepsilon(N+2s)-N}}
}\left (  1+ \frac{1}{\delta}\right )^{\frac{N+2s}{\varepsilon(N+2s)-N}} \right ]^{\frac{1}{\tx{r}-\tx{m}}},$$
inequality \eqref{Y0} is satisfied. Therefore, taking into account \eqref{0k}, \eqref{k} and \eqref{k3}, we obtain 
\begin{align} \label{stima2}
  \operatorname*{ess\,sup}_{Q_{\lambda\rr,\lambda\sigma}} u & \leq \frac{c}{(1-\lambda)^{\frac{(N+2s)^2}{\Lambda_\tx{r}}}\sigma^{\frac{N+2s}{\Lambda_\tx{r}}}}\left (  1+ \frac{1}{\delta}\right )^{\frac{N+2s}{\Lambda_\tx{r}}} \left ( \iint_{Q_{\rr,\sigma}} u^\tx{r}(x,t) \dx \dt \right )^{\frac{2s}{\Lambda_\tx{r}}} \notag \\ & \qquad + \left [ \delta \frac{\sigma}{\rr^{2s}}\left (\frac{R}{\rr}\right)^N \operatorname*{ess\,sup}_{t \in [-\sigma,0]} \mint_{B_R}u(x,t) \dx\right ]^{\frac{1}{q}} + \left ( \frac{\sigma}{\rr^{2s}} \right )^{\frac{1}{1-q}} \notag \\ & \qquad + c\left (\frac{R}{R-\rr} \right )^{\frac{N+2s}{q}} \textnormal{Tail}^{\frac{1}{q}}(u;Q_{R,\sigma}),
\end{align}
that holds for all $Q_{R,\sigma} \Subset \Omega_T$, $\lambda \in (0,1)$, $\delta>0$ and $\rr \in (0,R)$; with $c\equiv c(s,q,N,\tx{L},\tx{r})$.
Let us denote with ${\hat{\tx{R}}}_1$ the first term in the right-hand side above and for $n \in \mathbb{N}_0$, introduce again \eqref{ia}. Then applying \eqref{stima2}  with $(\lambda_n, \tilde{\rr}_n, \rr_{n+1}, \sigma_{n+1})$ in place of $(\lambda, \rr, R, \sigma)$, we estimate the last three terms in \eqref{stima2} as $\tilde{\tx{R}}_2, \tilde{\tx{R}}_3, \tilde{\tx{R}}_4$ above, while for $\hat{\tx{R}}_1$, we observe that, similarly to \eqref{delta1}, 
\begin{align} \label{deltaaa}
    \hat{\tx{R}}_1 & \leq c \left (\frac{2^n}{1-\lambda} \right )^{\frac{(N+2s)^2}{\Lambda_\tx{r}}} \left ( \frac{\rr^{2s}}{\sigma}\right )^{\frac{N}{\Lambda_\tx{r}}} \left (  1+ \frac{1}{\delta}\right )^{\frac{N+2s}{\Lambda_\tx{r}}} \left ( \mint\mint_{Q_{\rr,\sigma}} u^\tx{r}(x,t) \dx \dt \right )^{\frac{2s}{\Lambda_\tx{r}}}.
\end{align}
Therefore, using \eqref{deltaaa},\eqref{delta2}-\eqref{delta4}, we get
\begin{align*}
    \operatorname*{ess\,sup}_{Q_{\rr_n, \sigma_n}} u & \leq c \left (\frac{2^n}{1-\lambda} \right )^{\frac{(N+2s)^2}{\Lambda_\tx{r}}} \left ( \frac{\rr^{2s}}{\sigma}\right )^{\frac{N}{\Lambda_\tx{r}}} \left (  1+ \frac{1}{\delta}\right )^{\frac{N+2s}{\Lambda_\tx{r}}} \left ( \mint\mint_{Q_{\rr,\sigma}} u^\tx{r}(x,t) \dx \dt \right )^{\frac{2s}{\Lambda_\tx{r}}} \notag \\ & \qquad + c\delta \operatorname*{ess\,sup}_{Q_{\rr_{n+1},\sigma_{n+1}}} u + c\left ( \frac{\sigma}{\rr^{2s}}\right )^{\frac{1}{q-1}}\notag  \\ & \qquad + c \left ( \frac{2^{n}}{1-\lambda} \right )^{\frac{N+2s}{q}} \text{Tail}^{\frac{1}{q}}(u; Q_{\lambda \rr, \sigma}).
\end{align*}
Now, proceeding as in case 1, we obtain
\begin{align*}
    \operatorname*{ess\,sup}_{Q_{\lambda\rr, \lambda\sigma}} u & \leq c \left (\frac{1}{1-\lambda} \right )^{\frac{(N+2s)^2}{\Lambda_\tx{r}}} \left ( \frac{\rr^{2s}}{\sigma}\right )^{\frac{N}{\Lambda_\tx{r}}} \left ( \mint\mint_{Q_{\rr,\sigma}} u^\tx{r}(x,t) \dx \dt\right )^{\frac{2s}{\Lambda_\tx{r}}} \notag \\ &  \qquad + c\left ( \frac{\sigma}{\rr^{2s}}\right )^{\frac{1}{q-1}} + c \left ( \frac{1}{1-\lambda} \right )^{\frac{N+2s}{q}} \text{Tail}^{\frac{1}{q}}(u; Q_{\lambda \rr, \sigma}),
\end{align*}
for some $c\equiv c(s,q,N,\tx{L},\tx{r})$. Taking $\lambda=1/2$ we get \eqref{tesi1}. \\ \\ \noindent Take now $\tx{m}=q+1$. Consider again $Y_n$ as in \eqref{Yn}. As before, we get from \eqref{stima} with $\tx{r} = q+1$ the following relation
\eqn{follr}
$$Y_{n+1} \leq c \frac{2^{\alpha(n)}}{(1-\lambda)^{\frac{(N+2s)^2}{N}}\sigma^{\frac{N+2s}{N}}} \left (  1+ \frac{1}{\delta}\right )^{\frac{N+2s}{N}} Y_n^{1+\frac{2s}{N}}.$$
Therefore, applying the fast geometric convergence Lemma \ref{fg} we obtain that there exists a positive constant $c$, depending only on $s,q,N,\tx{L}$, such that $Y_n \to 0$ if 
\eqn{y0}
$$ Y_0 = \iint_{Q_0}\left(u(x,t)-\frac{k}{2} - \ell(t)\right)_+^{\tx{m}} \dx \dt \leq c^{-1} (1-\lambda)^{\frac{(N+2s)^2}{2s}}\sigma^{\frac{N+2s}{2s}} \left (  1+ \frac{1}{\delta}\right )^{-\frac{N+2s}{2s}}.$$
Recall \eqref{Y1}:
$$ Y_0 \leq \tx{D} \frac{2^{\tx{r}-\tx{m}}}{k^{\tx{r}-\tx{m}}},$$
then, inequality \eqref{y0} holds provided
\eqn{k4}
$$ k \geq 2 \left [ \tx{D} \frac{c}{(1-\lambda)^{\frac{(N+2s)^2}{2s}}\sigma^{\frac{N+2s}{2s}}} \left (  1+ \frac{1}{\delta}\right )^{\frac{N+2s}{2s}} \right ]^{\frac{1}{\tx{r}-\tx{m}}}.$$
Using that $\tx{m}-1=q=\frac{N+2s}{N-2s}$, we get $2s(\tx{r}-\tx{m})=\Lambda_\tx{r}$.
Taking into account \eqref{0k}, \eqref{k}, \eqref{k4}, we get \eqref{stima2}, with $c\equiv c(s,q,N,\tx{L},\tx{r})$, and so proceeding as before we arrive at \eqref{tesi1}. This ends the proof.
\end{proof}
\noindent
\noindent
Theorem~\ref{th2} continues the analysis in the spirit of \cite[Theorem 1.1]{C}. Without any a priori higher integrability of the solution, a H\"older-type estimate as in \eqref{Y1} cannot be performed, and the resulting bound remains qualitative.
\begin{proof}[Proof of Theorem \ref{th2}]
Let $u$ be a local, nonnegative, sub-solution of \eqref{pm} in $\Omega_T$. Observe that now \eqref{qc} implies $\tx{m}\geq q+1$.  Introducing as before \eqref{sis1}, we arrive at the iterative inequality \eqref{stima}. Now, if $\tx{m}>q+1$, we are in the setting of Theorem \ref{th1} and the quantitative estimate \eqref{tesi1} holds. It remains to analyze the case $\tx{m}=q+1$. Employing \eqref{stima} with $\tx{r} = q+1$, we have
\begin{align*}
     \iint_{Q_{n+1}} (u-k_{n+1}^t)_+^{\tx{m}} \dx \dt \leq c \frac{2^{\alpha(n)}}{(1-\lambda)^{\frac{(N+2s)^2}{N}}\sigma^{\frac{N+2s}{N}}}\left (  1+ \frac{1}{\delta}\right )^{\frac{N+2s}{N}} \left ( \iint_{Q_n}(u-k_{n}^t)_+^{q+1} \dx \dt\right )^{\frac{N+2s}{N}}.
\end{align*}
Setting
$$ Y_n := \iint_{Q_n} (u-k_n^t)_+^{\tx{m}} \dx \dt,$$
we observe that \eqref{follr} holds. Now, since
$$  Y_0 = \iint_{Q_0}\left(u(x,t)-\frac{k}{2} - \ell(t)\right)_+^{\tx{m}} \dx \dt \leq \iint_{Q_0 \cap \{ u \geq k/2\}}u^{\tx{m}}(x,t) \dx \dt$$
and
\begin{align*}
    \lim_{k \to \infty}|\{u \geq k/2\}| \leq \lim_{k\to \infty} \frac{2^\tx{m}}{k^{\tx{m}}} \iint_{Q_0} u^\tx{m}(x,t) \dx \dt =0,
\end{align*}
it holds,
$$ \lim_{k \to \infty}  \iint_{Q_0 \cap \{ u \geq k/2\}}u^{\tx{m}}(x,t) \dx \dt =0.$$
Therefore, choosing $k$ large enough such that \eqref{y0} holds, we get
$$  \operatorname*{ess\,sup}_{Q_{\rr/2, \sigma/2}} u \leq \frac{3}{2} k,$$
where we used also \eqref{0k}. The proof is complete.
    
\end{proof}

\end{document}